\documentclass[12pt,a4paper]{amsart}

\usepackage{amsthm, amssymb, amsmath}
\usepackage{amsopn}

\usepackage[a4paper]{geometry}
\usepackage[all]{xy}
\input{xypic}
\usepackage{tikz-cd}

\usepackage{mathrsfs}
\usepackage{enumitem}
\usepackage{mathtools}
\usepackage{url}

\usepackage{hyperref}

\advance \topmargin by -\headheight
\advance \topmargin by -\headsep
\evensidemargin \oddsidemargin 
\numberwithin{equation}{section}
\setenumerate[1]{label={\alph*)}} 

\renewcommand\S{§}

\newcommand{\al}{\alpha}
\newcommand{\Om}{\Omega}
\newcommand{\om}{\omega}
\newcommand{\D}{\Delta}

\newcommand{\La}{\Lambda}
\newcommand{\la}{\lambda}

\def\C{\mathbb{C}}

\newcommand{\Z}{\mathbb{Z}}

\newcommand{\CA}{\mathcal{A}}

\newcommand{\CD}{\mathcal{D}}
\newcommand{\E}{\mathcal{E}}

\newcommand{\OO}{\mathcal{O}}

\newcommand{\g}{\mathfrak{g}}

\newcommand{\ov}[1]{\overline{#1}}
\newcommand{\what}[1]{\widehat{#1}}
\newcommand{\wtil}[1]{\widetilde{#1}}

\DeclareMathOperator{\gr}{gr}
\newcommand{\filf}[1]{{F^{#1}}}
\newcommand{\grf}[1]{{\gr_F^{#1}}}
\newcommand{\filg}[1]{{G_{#1}}}
\newcommand{\grg}[1]{{\gr^G_{#1}}}
\newcommand{\filw}[1]{{W_{#1}}}

\newcommand{\file}[1]{{E^{#1}}}
\newcommand{\gre}[1]{{\gr_E^{#1}}}

\newcommand{\X}{\Xi}

\newcommand{\vac}{\mathbf{1}}

\newcommand{\Hch}{H^{\text{ch}}}
\newcommand{\dR}{\text{dR}}
\newcommand{\ovy}{\overline{Y}}
\newcommand{\ovX}{X}

\DeclareMathOperator{\Ran}{Ran}

\DeclareMathOperator{\Hoch}{HH}

\DeclareMathOperator{\res}{res}

\DeclareMathOperator{\zhu}{Zhu}
\DeclareMathOperator{\HP}{HP}

\swapnumbers
\newtheoremstyle{exps}{\topsep}{\topsep}{}{0pt}{\bfseries}{.}{0pt}{}

\newtheorem*{thm*}{Theorem}
\newtheorem*{prop*}{Proposition}
\newtheorem*{lem*}{Lemma}
\newtheorem*{cor*}{Corollary}
\newtheorem*{rem*}{Remark}
\newtheorem{thm}{Theorem}[section]
\newtheorem{prop}[thm]{Proposition}

\theoremstyle{definition}
\newtheorem*{defn*}{Definition}
\newtheorem*{exer*}{Exercise}

\newtheorem{defn}[thm]{Definition}
\newtheorem*{problem*}{Problem}
\newtheorem{rem}[thm]{Remark}
\newtheorem{nolabel}[thm]{ }

\theoremstyle{exps}

\title[]{A Hodge filtration on chiral homology and {P}oisson homology of associated schemes}

\begin{document}

\begin{center}
{\LARGE \bf A Hodge filtration on chiral homology and Poisson homology of associated schemes} \par \bigskip

\renewcommand*{\thefootnote}{\fnsymbol{footnote}}
{\normalsize
	Jethro van Ekeren\footnote{email: \texttt{jethro@impa.br}} and
	Reimundo Heluani
}

\par \bigskip
{\footnotesize Instituto de Matem\'{a}tica Pura e Aplicada \\
Rio de Janeiro, RJ, Brazil}

\par \bigskip

\vspace*{10mm}

\emph{Dedicated to our teacher, Victor G. Kac, on the occasion of his 80th birthday}

\end{center}

\vspace*{10mm}

\noindent
\textbf{Abstract.} We introduce filtrations in chiral homology complexes of smooth elliptic curves, exploiting the mixed Hodge structure on cohomology groups of configuration spaces. We use these to relate the chiral homology of a smooth elliptic curve with coefficients in a vertex algebra with the Poisson homology of the associated Poisson scheme. As an application we deduce finite dimensionality results for chiral homology in low degrees.

\vspace*{3mm}

{\footnotesize\texttt{MSC Identifiers: 17B69; 14H52; 32S35; 17B63}}

\vspace*{10mm}

\section{Introduction}

\begin{nolabel}
Let $X$ be a smooth complex algebraic curve and $x_1, \ldots, x_n$ a sequence of marked points of $X$. Let $\g$ be a finite dimensional simple Lie algebra and $\widehat{\g}$ the corresponding affine Lie algebra, and let $M_1, \ldots, M_n$ be a sequence of $\what\g$-modules of a given level $k$. From these data a certain vector space of coinvariants may be constructed, namely the coinvariants of $\bigotimes_{i=1}^n M_i$ relative to the action of the Lie subalgebra $\what{\g}_{\text{out}} \subset \what\g$ determined by functions with poles only at $x_1, \ldots x_n$.
\end{nolabel}

\begin{nolabel}
If the $\what\g$-modules $M_i$ are subject to appropriate hypotheses, for example that each one be finitely generated and integrable, then the space of coinvariants is finite dimensional. Furthermore the dual vector spaces, considered for varying $x_1, \ldots, x_n$ and varying $X$, assemble to form vector bundles with connection, known as conformal blocks, over moduli spaces of smooth pointed curves \cite{zhu96, TUY, DGT}. Conformal blocks, especially their extension to moduli spaces of stable curves, have been intensively studied due to their significance in both algebraic geometry and representation theory.
For instance braiding and associativity morphisms for the Kazhdan-Lusztig tensor product are defined in terms of monodromy and asymptotics of flat sections of these bundles \cite{KL1-4}.
\end{nolabel}

\begin{nolabel}
The theory of vertex algebras and chiral algebras provides a natural and more general context for the conformal blocks construction. For example the case described above is recovered upon specialisation to the case of the simple vertex algebra $L_k(\g)$.
\end{nolabel}

\begin{nolabel}
The Zhu Poisson algebra $R_V$ of a vertex algebra $V$ is the quotient vector space $R_V = V / V_{(-2)}V$, with commutative product induced by the $(-1)$-product in $V$ and Poisson bracket induced by the $(0)$-product in $V$. The vertex algebra $V$ is said to be $C_2$-cofinite precisely if $\dim R_V < \infty$ \cite{zhu96}. The $C_2$-cofiniteness condition of Zhu is related to integrability in the case of $L_k(\g)$ \cite{DM06} and is a useful sufficient condition for finite dimensionality of spaces of coinvariants in general \cite[Theorem 4.7]{AN} \cite[Proposition 5.1.1]{DGT}. Zhu proved modularity of characters of rational $C_2$-cofinite vertex algebras by studying the conformal blocks of smooth elliptic curves. It was later observed in \cite{Arakawa.Kawasetsu} that in this particular case, finite dimensionality of coinvariants is ensured by a condition much weaker than finite dimensionality of $R_V$, namely finite dimensionality of the Poisson homology $\HP_0(R_V) = R_V / \{R_V, R_V\}$.
\end{nolabel}

\begin{nolabel}
The notions of chiral algebra, and chiral homology $\Hch_n(X, \CA)$ of an algebraic curve $X$ with coefficients in a chiral algebra $\CA$ over $X$, have been introduced by Beilinson and Drinfeld \cite{BD}. Spaces of coinvariants as above are recovered as the degree $0$ chiral homology. In our previous article \cite{EH21} we have worked out these constructions in the case of $X = E$ a smooth elliptic curve and $\CA_V$ the chiral algebra over $E$ associated with a vertex algebra $V$, and we have related them to the Hochschild homology of the associative algebra $\zhu(V)$. The main result of that article was that for a classically free vertex algebra $V$ the limit of $\Hch_1(E, \CA_V)$, as $E$ degenerates to a nodal curve, recovers $\Hoch_1(\zhu(V))$ the Hochschild homology of $\zhu(V)$ and in particular vanishes if $\zhu(V)$ is semisimple.
\end{nolabel}

\begin{nolabel}
The purpose of this article is to establish a link, valid in all degrees, between the chiral homology complex of a vertex algebra $V$ and the Lichnerowicz complex computing Poisson homology of $R_V$. This is done through the introduction of an appropriate filtration on the chiral complex and passage to its associated spectral sequence. In prior work \cite{EH21} we have used filtrations on the vertex algebra $V$ itself to approximate the discrepancy between $\Hch_1(E, \CA_V)$ and $\Hoch_1(\zhu(V))$, and the condition of classical freedom of $V$ arises in this context as a sufficient condition on $V$ guaranteeing vanishing of the discrepancy. In this article the filtration on $V$ is enhanced to a filtration on the whole chiral complex using the mixed Hodge structure on the cohomology groups of configuration spaces of the elliptic curve $E$. With this refinement, the relation between $\Hch_0(E, \CA_V)$ and $\HP_0(R_V)$ remarked upon above, is promoted to arbitrary degree. A precise statement is given in Theorem \ref{thm:chiral.contains.HP} below, and some consequences are deduced in Section \ref{sec:consequences}.
\end{nolabel}

\begin{nolabel}
\emph{Acknowledgements.} The authors would like to thank T. Arakawa, J. Francis, B. Pym, P. Safranov, T. Schedler and M. Szczesny for helpful discussions. JvE has been supported by Serrapilheira Institute grant Serra -- 2023-0001, FAPERJ grant 201.445/2021 and CNPq grant 310576/2020-2 and RH by FAPERJ CNE/2019 and CNPq grant 05688/2019-7.
\end{nolabel}

\section{Preliminaries on vertex algebras}

\begin{nolabel}
In this section we collect definitions and background material related to vertex algebras. Throughout the paper, vector spaces are over the field $\C$ of complex numbers, $\otimes$ means $\otimes_\C$ unless otherwise specified, and by algebra we mean $\C$-algebra. We use angle brackets $\left<\cdots\right>$ as notation for the $\C$-linear span of enclosed terms.
\end{nolabel}

\begin{defn}\label{def:va.defn}
A vertex algebra is a vector space $V$ equipped with a set of bilinear products $V \times V \rightarrow V$, written $(a, b) \mapsto a_{(n)}b$ and indexed by the integers $n \in \Z$, and a nonzero vector $\vac \in V$ called the vacuum vector. The linear endomorphisms $a_{(n)}$ associated with $a \in V$, are called the modes of $a$. The formal series
\begin{align*}
Y(a, z) = \sum_{n \in \Z} a_{(n)} z^{-n-1}
\end{align*}
is called the field associated with $a$. The following axioms are to be satisfied. 
\begin{enumerate}
\item[(1)] (Quantum field property) For fixed $a, b \in V$, there exists $N$ such that $a_{(n)}b = 0$ for all $n \geq N$, i.e., $Y(a, z)b \in V((z))$,

\item[(2)] (Unit axiom) $Y(\vac, z) = I_V$ and $Y(a, z)\vac \in V[[z]]$ with $Y(a, z)\vac|_{z=0} = a$, for all $a \in V$,

\item[(3)] (Borcherds identity) For all $a, b, c \in V$ and $m, n, k \in \Z$ the following relation holds:
\begin{align}\label{eq:bor.iden}
\begin{split}
& \sum_{j \in \Z_{\geq 0}} \binom{m}{j} (a_{(n+j)}b)_{(m+k-j)} c \\
= { } & \sum_{j \in \Z_{\geq 0}} (-1)^j \binom{n}{j} \left( a_{(n+m-j)} b_{(k+j)} c - (-1)^n b_{(n+k-j)} a_{(m+j)} c \right).
\end{split}
\end{align}
\end{enumerate}
\end{defn}

\begin{nolabel}
There are alternative definitions of the notion of vertex algebra, involving structures which can in turn be recovered from Definition \ref{def:va.defn}. For example the translation operator $T : V \rightarrow V$ is given by $T(a) = a_{(-2)}\vac$, it satisfies the identity
\begin{align*}
[T, Y(a, z)] = \partial_z Y(a, z)
\end{align*}
for all $a \in V$. A useful consequence of the axioms is the skew-symmetry relation
\begin{align}\label{eq:skew-symm}
Y(b, z)a = e^{zT} Y(a, -z)b,
\end{align}
valid for all $a, b \in V$.
\end{nolabel}

\begin{defn}
A conformal vertex algebra of central charge $c$ is a vertex algebra $V$ equipped with a $\Z_{\geq 0}$-grading $V = \bigoplus_{\Delta \in \Z_{\geq 0}} V_\Delta$ and a distinguished vector $\om \in V_2$ whose modes
\begin{align*}
Y(\om, z) = L(z) = \sum_{n \in \Z} L_n z^{-n-2}
\end{align*}
furnish $V$ with a representation of the Virasoro Lie algebra as follows
\begin{align*}
[L_m, L_n] = (m-n)L_{m+n} + \delta_{m, -n} \frac{m^3-m}{12} c I_V.
\end{align*}
Furthermore this representation is to be compatible with the vertex algebra structure in the sense that $V_\D = \ker(L_0 - \D I_V)$ for all $\D \in \Z_{\geq 0}$, and $L_{-1} = T$.
\end{defn}

\begin{nolabel}
If we write $\D$ for the degree of a homogeneous element of the conformal vertex algebra $V$, i.e., $\D(a) = n$ for $a \in V_n$, then
\begin{align}\label{eq:graded.VA.formula}
\D(a_{(n)}b) = \D(a) + \D(b) - n - 1.
\end{align}
In particular $\D(Ta) = \D(a)+1$.
\end{nolabel}

\begin{defn}
A graded vertex algebra is a vertex algebra $V$ equipped with a $\Z_{\geq 0}$-grading $V = \bigoplus_{n \in \Z_{\geq 0}} V_n$ for which (1) $\vac \in V_0$, (2) $T : V_n \rightarrow V_{n+1}$, and (3) writing $\D(a) = n$ for $a \in V_n$ as above, the relation \eqref{eq:graded.VA.formula} holds for all homogeneous $a, b \in V$.
\end{defn}

\begin{nolabel}
The Borcherds identity \eqref{eq:bor.iden} can be cast in the following form:
\begin{align}\label{eq:bor.iden.2}
\begin{split}
& \res_{z=0} \res_{w=0} Y(a, z) Y(b, w) c \, i_{z, w} f(z, w) - \res_{z=0} \res_{w=0} Y(b, w) Y(a, z) c \, i_{w, z} f(z, w) \\
& = \res_{w=0} \res_{z-w=0} Y(Y(a, z-w)b, w) c \, i_{w, z-w} f(z, w)
\end{split}
\end{align}
for all
\begin{align*}
f(z, w) = z^m w^k (z-w)^n.
\end{align*}
For notational background on $\res_{z=0}$ and $i_{z, w}$, etc., we refer the reader to the book \cite{kac.book} or to any textbook on vertex algebras. We give a brief description for the sake of completeness. The formal residue symbol $\res_{z=0}$ means extraction of the coefficient of $z^{-1}$. The symbol $i_{z, w}$ means expansion of $f(z, w)$ as a Laurent series in $w$, more precisely, expansion in the field $\C((z))((w))$. Similarly $i_{w, z-w}$ represents expansion in $\C((w))((z-w))$.
\end{nolabel}

\begin{nolabel}
Clearly \eqref{eq:bor.iden.2} holds for all $f(z, w) \in \C[z^{\pm 1}, w^{\pm 1}, (z-w)^{-1}]$ by linearity. We remark that, at least when $V$ is a graded vertex algebra, this identity holds more generally for
\begin{align}\label{eq:Laurent.series.ring}
f(z, w) \in \C[[z, w]][z^{-1}, w^{-1}, (z-w)^{-1}].
\end{align}
It is straightforward to see why: fix such an expression $f(z, w)$ and let $M$ and $K$ be, respectively, the minimal exponents of $z$ and $w$ that appear in it. Also fix $a, b, c \in V$ and $n \in \Z$. There exists $A \in \Z$ such that both sides of \eqref{eq:bor.iden} vanish when $m, k \geq A$. For each $m \in \Z$ such that $M \leq m \leq A$ there exists $A' \in \Z$ such that for $k \geq A'$ both sides again vanish. If we assume $V$ to be graded then the same is true with the roles of $m$ and $k$ exchanged, since for fixed $a, b, c \in V$ and fixed $n, k \in \Z$ the conformal weight of $(a_{(n+j)}b)_{(m+k-j)}c$ becomes negative as $m$ becomes large and positive. It follows that only a finite number of the coefficients of $f(z, w)$ contribute nontrivially to \eqref{eq:bor.iden.2}. Thus this apparently more general form of Borcherds identity follows from the standard form.
\end{nolabel}

\begin{nolabel}
For the purposes of stating the Borcherds identity in contour integral form, we denote by $C_z^a(r)$ the contour $z = a + r e^{it}$ for $0 \leq t \leq 2\pi$ in the complex $z$-plane. Now let $0 < \varepsilon < r$. The contour integral formulation of Borcherds identity (see {\cite[Section 3.3.10]{FBZ}} for instance) asserts
\begin{align}\label{eq:bor.contour}
\begin{split}
& \oint_{C_z^{0}(r+\varepsilon)} \oint_{C_w^0(r)} f(z, w) Y(a, z) Y(b, w) c \, dw \, dz - \oint_{C_w^{0}(r)} \oint_{C_z^0(r-\varepsilon)} f(z, w) Y(b, w) Y(a, z) c \, dz \, dw \\
= {} & \oint_{C_w^0(r)} \oint_{C_z^{w}(\varepsilon)} f(z, w) Y(Y(a, z-w)b, w)c \, dz \, dw,
\end{split}
\end{align}
for all $a, b, c \in V$ and $f(z, w)$ of the form $z^m w^k (z-w)^n$. As we have seen above, the formal analogue \eqref{eq:bor.iden.2} continues to hold for $f(z, w) \in \C[[z, w]][z^{-1}, w^{-1}, (z-w)^{-1}]$, even if this series has zero radius of convergence. If $f(z, w)$ is meromorphic with poles on $z=0$, $w=0$ and $z=w$, then it has a series expansion as in  \eqref{eq:Laurent.series.ring} and the three terms in \eqref{eq:bor.contour} make sense and the equality is verified.
\end{nolabel}

\begin{nolabel}
We recall some material about filtrations on (conformal) vertex algebras from \cite{Li.filtration} and \cite{Arakawa.C2}. Let $V$ be a vertex algebra. A set $\{a^i \mid i \in I\} \subset V$ of nonzero vectors $a^i$ is said to be a strong generating set, if $V$ is spanned by the vectors
\begin{align}\label{eq:monom}
a^{i_1}_{(-n_1-1)} a^{i_2}_{(-n_2-1)} \cdots a^{i_s}_{(-n_s-1)}\vac,
\end{align}
where $i_j \in I$ and $n_j \in \Z_{\geq 0}$ for $1 \leq j \leq s$.
\end{nolabel}

\begin{defn}
Let $V$ be a vertex algebra. An increasing filtration of $V$ by vector subspaces $\filg{p}V$ indexed by $p \in \Z$, is said to be a \emph{good increasing filtration} if (1) it is exhaustive (each vector $v \in V$ is contained in $\filg{p}V$ for some $p \in \Z$), and (2) for all $a \in \filg{p}V$ and $b \in \filg{q}V$ we have $a_{(n)}b \in \filg{p+q}V$ for all $n \in \Z$ and $a_{(n)}b \in \filg{p+q-1}V$ for all $n \in \Z_{\geq 0}$.
\end{defn}

\begin{defn}
Let $V$ be a conformal vertex algebra, and let $\{a^i \mid i \in I\}$ be a strong generating set for $V$, in which all $a^i$ are homogeneous with respect to the conformal structure of $V$. Denoting by $\filg{p}V$ the span in $V$ of all vectors of the form \eqref{eq:monom}, in which $\D(a^{i_1}) + \cdots + \D(a^{i_s}) \leq p$, defines a good increasing filtration on $V$. This is called the standard filtration on $V$ relative to the chosen strong generating set.
\end{defn}

\begin{defn}[The Li filtration {\cite{Li.abelianizing}}]
Let $V$ be a vertex algebra. Denote by $\filf{p}V$ the span in $V$ of all vectors of the form
\begin{align*}
a^{i_1}_{(-n_1-1)} a^{i_2}_{(-n_2-1)} \cdots a^{i_s}_{(-n_s-1)}b,
\end{align*}
(here the $a^{i}$ and $b$ are arbitrary elements of $V$) in which $n_1 + \cdots + n_s \geq p$. This defines a decreasing filtration on $V$.
\end{defn}

\begin{nolabel}
If $\filg{}$ is a good increasing filtration on $V$ then the associated graded $\grg{}(V)$ acquires the structure of a commutative algebra. The commutative associative product is that induced from $(a, b) \mapsto a_{(-1)}b$, and the unit is the image in $\grg{0}(V)$ of the vacuum vector $\vac$. The translation operator $T$ induces a derivation of degree $0$ on $\grg{}(V)$. As for the Li filtration, the associated graded $\grf{}(V)$ becomes a commutative algebra with derivation in the same way, only now the derivation $T$ has degree $+1$.
\end{nolabel}

\begin{nolabel}\label{nolabel:def.RV}
The vector space $\grf{0}(V)$ is a subalgebra of the commutative algebra $\grf{}(V)$, not closed under the derivation $T$. In fact $\filf{1}V = V_{(-2)}V$, and the commutative algebra $\grf{0}(V) = V / \filf{1}V = V / V_{(-2)}V$ carries a natural structure of Poisson algebra, with Poisson bracket given by $\{a, b\} = a_{(0)}b$ \cite{zhu96}. This is known as the Zhu Poisson algebra of $V$ and we denote it $R_V$.
\end{nolabel}

\begin{nolabel}
Both of the associated graded algebras $\grf{}(V)$ and $\grg{}(V)$ carry natural structures of Poisson vertex algebra. The definition of this notion, which we do not reproduce here, can be found in \cite{Li.filtration}. It was shown by Arakawa \cite{Arakawa.C2} that $\filf{p}V_n = \filg{n-p}V_n$ for all $p$ and $n$, and that
\begin{align*}
\grf{}(V) \cong \grg{}(V)
\end{align*}
as Poisson vertex algebras. In particular all standard filtrations on $V$ are good filtrations, and they are all equivalent \cite{Li.filtration}.
\end{nolabel}

\begin{nolabel}
We recall the arc scheme construction (or rather ``arc algebra'' construction, since we work solely with coordinate rings rather than their associated affine schemes). Informally the arc algebra $(JR, \partial)$ of a commutative $\C$-algebra $R$ is obtained by adjoining a derivation $\partial$ to $R$ in the freest possible way. If $R$ is presented as a polynomial algebra in variables $x_1, \ldots, x_n$, modulo the ideal generated by polynomials $f_1, \ldots, f_s$, then $JR$ is the algebra of polynomials in $x_i, \partial x_i, \partial^2 x_i, \ldots$ for $i=1, \ldots, n$, modulo the ideal generated by $f_j, \partial f_j, \partial^2 f_j, \ldots$ for $j=1, \ldots, s$. Here $\partial{f_j}$, etc., is interpreted via formal application of the Leibniz rule: for example $\partial(x^2 y) = 2x y \partial{x} + x^2 \partial{y}$, etc. Finally $\partial : JR \rightarrow JR$ is the unique derivation satisfying $\partial(\partial^kx_i) = \partial^{k+1}{x_i}$.
\end{nolabel}

\begin{nolabel}
If $R$ is a Poisson algebra, then the arc algebra $JR$ acquires a natural structure of Poisson vertex algebra with $\partial$ as translation operator. There exists therefore a canonical morphism
\begin{align*}
JR_V \rightarrow \grf{}(V)
\end{align*}
of commutative algebras with derivation, which is in fact a morphism of Poisson vertex algebras. By \cite[Lemma 4.1]{Li.abelianizing} this morphism is always surjective. If this morphism is also injective, we say that $V$ is ``classically free'' \cite[Section 1.1]{AEH23}. The condition of classical freedom has a simple homological interpretation, which was given in {\cite[Section 13]{EH21}}. First we recall the general construction of K\"{a}hler differentials.
\end{nolabel}

\begin{defn}\label{defn:Kahler}
Let $k$ be a field and $R$ a commutative $k$-algebra. The K\"{a}hler differentials of $R$ denotes the $R$-module
\begin{align*}
\Omega_{R} = \Omega_{R/k} = (R \otimes_k R) / S, \qquad S = \left< ax \otimes y + ay \otimes x - a \otimes xy \mid a, x, y \in R \right>.
\end{align*}
The action of $R$ on $\Omega$ is by multiplication on the first tensor factor. The image of $a \otimes x$ is denoted $a \, dx$.
\end{defn}

\begin{nolabel}\label{nolabel:Koszul.complex}
Denote by $A$ the associated graded algebra $\grf{}(V)$ and $T$ the derivation of degree $+1$ which it inherits. We apply the construction of K\"{a}hler differentials to $R = A$, considered as a $\C$-algebra, and then form the exterior algebra $\Omega_A^\bullet = \wedge^\bullet_A \Omega_A$. The extra structure of the derivation $T$ allows us to turn $\Omega_A^\bullet$ into a homological complex, with differential $d$ defined to be the unique derivation satisfying
\begin{align*}
d(b \, da) = (Ta) \cdot b.
\end{align*}
In other words $d$ is contraction with the vector field $T$, and is explicitly given by
\begin{align*}
d(b \, da^1 \wedge da^2) = (Ta^1) b \, da^2 - (Ta^2)b \, da^1,
\end{align*}
etc. We call this the Koszul complex of the pair $(A, T)$. We have the following result.
\end{nolabel}

\begin{thm}[{\cite[Theorem 13.7]{EH21}}]\label{thm:classically.free.koszul}
Let $V$ be a finitely strongly generated vertex algebra and let $A = \grf{}(V)$ and $T$ be as above. The first Koszul homology $H_1(A, T) = H_1(\Om_A^\bullet, d)$ vanishes if and only if $V$ is classically free.
\end{thm}

\section{Prelude on Leibniz algebras}\label{sec:Leibniz}

\begin{nolabel}
In the next section we introduce a homological complex computing a variant of the Beilinson-Drinfeld chiral homology of a smooth elliptic curve. Before introducing and studying this complex, which is essentially the Chevalley-Eilenberg complex of a Lie algebra, it will be instructive to examine a generalisation of Lie algebra homology to the class of Leibniz algebras, introduced by Loday.
\end{nolabel}

\begin{defn}
A Leibniz algebra is a vector space $\g$ equipped with a bilinear operation $[,] : \g \times \g \rightarrow \g$ satisfying the Jacobi identity
\begin{align}\label{eq:Leiniz.alg.def}
[x, [y, z]] = [[x, y], z] + [y, [x, z]]
\end{align}
for all $x, y, z \in \g$.
\end{defn}

\begin{nolabel}
A Lie algebra is of course a Leibniz algebra in which the skew-symmetry identity $[b, a] = -[a, b]$ holds. Apparently Loday \cite{Loday.93} was the first to observe that the Chevalley-Eilenberg complex $(\wedge^\bullet \g, d)$ for homology of a Lie algebra $\g$ admits a lift to a complex $(\g^{\otimes \bullet}, d)$ which makes sense even without the skew-symmetry identity.
\end{nolabel}

\begin{nolabel}
Let $(\g, [,])$ be a Leibniz algebra. We equip the vector spaces $C_n = \g^{\otimes n}$ with the structure of a homological complex via the differentials
\begin{align*}
d = \sum_{1 \leq i < j \leq n} (-1)^{n-i} d^{(ij)}, 
\end{align*}
where
\begin{align*}
d^{(ij)}\left( a^1 \otimes \cdots \otimes a^n \right)= a^1 \otimes \cdots \otimes \widehat{a}^i \otimes \cdots \otimes a^{j-1} \otimes [a^i, a^j] \otimes a^{j+1} \otimes \cdots \otimes a^n.
\end{align*}
This definition can also easily be generalised to encompass an action of $\g$ on a module. If $\g$ is a Lie algebra then we can consider the action of the symmetric group $\Sigma_n$ on $C_n$ by the sign permutation representation. The differential $d$ defined above, is well defined upon passage to the spaces of coinvariants $(C_n)_{\Sigma_n}$, recovering the Chevalley-Eilenberg complex $(\wedge^\bullet\g, d)$.
\end{nolabel}

\begin{nolabel}
Although it is a routine computation, we sketch the proof that $d \circ d = 0$ in Loday's complex. First consider the component of the composition $C_{n+2} \rightarrow C_n$ in which factor $a^i$ is applied to factor $a^j$ and then factor $a^k$ is applied to factor $a^\ell$ (where all four indices are distinct). This is $\pm 1$ times
\begin{align}\label{eq:crossing}
a^1 \otimes \cdots \otimes [a^{i}, a^{j}] \otimes \cdots  \otimes [a^{k}, a^{\ell}] \otimes \cdots \otimes a^{n+2}.
\end{align}
The same term, up to a sign, appears when performing the same two operations in the opposite order. We wish to see that it appears with opposite sign in fact. Take, without loss of generality, $i < k$, then the first composition appears with sign $(-1)^{(n+2-i) + (n+1-k)}$ while the second composition has sign $(-1)^{(n+2-k) + (n-i)}$ since, effectively, the block $(a^{i}, \ldots, a^{n+2})$ is shortened by one upon merging of $a^{k}$ with $a^{\ell}$. Thus the two terms cancel, as required.

It remains to check that, for $i < j < k$, the three terms $[[a^i, a^j], a^k]$, $[a^i, [a^j, a^k]]$ and $[a^j, [a^i, a^k]]$ cancel. Indeed these terms appear with signs 
\begin{align*}
(-1)^{(n+2-i)+(n+1-j)}, \quad (-1)^{(n+2-j)+(n-i)} \quad \text{and} \quad (-1)^{(n+2-i)+(n+1-j)},
\end{align*}
respectively, and thus cancel precisely because of \eqref{eq:Leiniz.alg.def}.
\end{nolabel}

\section{A chiral homology complex}

\begin{nolabel}
In this section we attach a homological complex to the data of a graded vertex algebra $V$ and a smooth elliptic curve $E$. If $V$ is conformal or quasiconformal then, via the construction presented in the book \cite{FBZ}, one obtains a chiral algebra over $E$ invariant under arbitrary local conformal transformations. Indeed the construction yields a chiral algebra over each smooth complex algebraic curve in a uniform way. Our complex is essentially the Beilinson-Drinfeld chiral homology complex of $E$ with coefficients in this chiral algebra \cite{BD} (see Remark \ref{rem:comparison.BD} below for further details). We remark that our complex makes sense even if $V$ is not quasiconformal. In this more general situation there is no longer an associated conformally invariant chiral algebra, but one still has a translation invariant chiral algebra on $E$. Our complex may then be interpreted geometrically as the chiral homology complex of this translation invariant chiral algebra over $E$.
\end{nolabel}

\begin{nolabel}\label{nolabel:algebraic.coords.E}
Let $\tau$ be a complex number with positive imaginary part, $\La \subset \C$ the lattice $\Z + \Z \tau$, and let $E$ be the smooth complex variety $\C / \La$ with marked point $0$. This is a smooth complex curve of genus $1$ with a marked point, in other words an elliptic curve. As a projective variety $E$ is the locus of
\begin{align*}
V^2 Z = 4 U^3 - 60G_4(\tau) U Z^2 - 140G_6(\tau) Z^3
\end{align*}
in homogeneous coordinates $[U, V, Z]$ on $\C P^2$. Here the $G_{2k}(\tau)$ are Eisenstein series, given by the following series expansion in general:
\begin{align*}
G_{2k}(\tau) = (-1)^{k+1} \frac{(2\pi)^{2k} B_{2k}}{(2k)!} \left(1 - \frac{4k}{B_{2k}} \sum_{n=1}^\infty \frac{n^{2k-1}q^n}{1-q^n} \right), \qquad q = e^{2\pi i \tau}.
\end{align*}
An explicit isomorphism between the models is given by $x \mapsto [u : v : 1]$ for $x \neq 0$, where $u = \wp(x)$ and $v = \wp'(x)$, and by $0 \mapsto [0 : 1 : 0]$. Here the Weierstrass function $\wp(x)$ is the $\Lambda$-periodic meromorphic function on $\C$ whose Laurent series expansion at $x=0$ is given by
\begin{align*}
\wp(x) = x^{-2} + \sum_{k = 1}^\infty (2k+1) G_{2k+2}(\tau) x^{2k}.
\end{align*}
The Weierstrass $\wp$-function has poles of order $2$ on the points of $\Lambda$.

The canonical bundle $\Omega_E$ of $E$ is trivial with global section $dx = du / v$ and the corresponding global section of the tangent bundle $\Theta_E$ is the derivation
\begin{align*}
\partial_x = v \partial_u + \frac{P'(u)}{2} \partial_v,
\end{align*}
where $P(u) = 4 u^3 - 60G_4(\tau) u - 140G_6(\tau)$. We denote by $\ovX$ the affine algebraic curve $E \backslash \{0\}$, and by $\Gamma_n$ the coordinate ring of the configuration space
\begin{align*}
F(\ovX, n) = \{(x_1, \ldots, x_n) \in {\ovX}^n \mid \text{$x_i \neq x_j$ for all $i \neq j$}\}.
\end{align*}
We denote by $\partial_i$, for $i = 1, \ldots, n$, the pullback of $\partial_x$ along the $i^{\text{th}}$ projection, considered as a derivation of $\Gamma_n$. 
\end{nolabel}

\begin{nolabel}
Let $V$ be a graded vertex algebra. We consider the vector space $\Gamma_n \otimes V^{\otimes n+1}$. For $i = 1, \ldots n$, we write $T_i$ for the action of the translation operator on the $i^{\text{th}}$ tensor factor of $V^{\otimes n+1}$. We consider the vector spaces
\begin{align}\label{eq:C.complex.def}
\widetilde{C}_n = \Gamma_n \otimes V^{\otimes n+1}, \quad \text{and} \quad C_n = \widetilde{C}_n / \sum_{i=1}^n (\partial_i + T_i) \widetilde{C}_n.    
\end{align}
We shall write down linear maps $d : \widetilde{C}_n \rightarrow \widetilde{C}_{n-1}$, check that they descend to well defined maps $d : {C}_n \rightarrow {C}_{n-1}$, and finally prove that $({C}_\bullet, d)$ is a complex.

We set
\begin{align}\label{eq:general.diff.def}
d = \sum_{1 \leq i < j \leq n} (-1)^{n-i} d^{(ij)} + \sum_{1 \leq i \leq n} (-1)^{n-i} p^{(i)}
\end{align}
where
\begin{align}\label{eq:general.diff.def.comp1}
\begin{split}
d^{(ij)} & \left( f(x_1, \ldots, x_n) a^1 \otimes \cdots \otimes a^n \otimes b\right) \\
= {} & \res_{x_i=x_j} f(x_1, \ldots, x_n) a^1 \otimes \cdots \otimes \widehat{a}^i \otimes \cdots \\
&\left. \cdots \otimes Y(a^i, x_i-x_j) a^j \otimes \cdots \otimes a^n \otimes b\right|_{(x_1, \ldots, \widehat{x}_i, \ldots, x_n) = (x_1, \ldots, x_{n-1})}
\end{split}
\end{align}
and
\begin{align}\label{eq:general.diff.def.comp2}
\begin{split}
 p^{(i)} & \left(f(x_1, \ldots, x_n) a^1 \otimes \cdots \otimes a^n \otimes b\right) \\
= {} & \res_{x_i=0} f(x_1, \ldots, x_n)  \left. a^1 \otimes \cdots \otimes \widehat{a}^i \otimes \cdots \otimes a^n \otimes Y(a^i, x_i) b \right|_{(x_1, \ldots, \widehat{x}_i, \ldots, x_n) = (x_1, \ldots, x_{n-1})}.
\end{split}
\end{align}
For instance in degrees $1$ and $2$ we have
\begin{align*}
d\left( f(x_1) a \otimes b \right) = {} & \res_{x_1=0} f(x_1) Y(a, x_1)b, \\
\text{and} \quad d\left( f(x_1, x_2) a^1 \otimes a^2 \otimes b \right)
= {} & -\res_{x_1=0} f(x_1, x_2) a^2 \otimes Y(a^1, x_1)b |_{x_2=x_1} \\
&+ \res_{x_2=0} f(x_1, x_2) a^1 \otimes Y(a^2, x_2)b |_{x_1=x_1} \\
&- \res_{x_1=x_2} f(x_1, x_2) Y(a^1, x_1-x_2) a^2 \otimes b |_{x_2=x_1}.
\end{align*}
\end{nolabel}

\begin{nolabel}
Before confirming that $d$ descends to the quotient $C_\bullet$, it is worth clarifying the definition of $d$ in $\wtil{C}_\bullet$. In particular we should confirm that the residues appearing in \eqref{eq:general.diff.def.comp1} and \eqref{eq:general.diff.def.comp2} all lie, in fact, in $\Gamma_{n-1}$. Consider for example
\begin{align*}
\res_{x_i=x_j} f(x_i, x_j) Y(a^i, x_i-x_j) a^j,
\end{align*}
where, for notational clarity, we suppress the dependence on variables other than $x_i$, $x_j$. We have the Laurent series expansion
\begin{align}\label{eq:i-expansion}
f(x_i, x_j) = \sum_{m=N}^\infty f^{(m)}(x_j) (x_i-x_j)^m
\end{align}
where
\begin{align}\label{eq:fm-integral}
f^{(m)}(x_j) = \frac{1}{2\pi i} \oint_\gamma \frac{f(x_i, x_j)}{(x_i-x_j)^{m+1}} \, dx_i
\end{align}
Here $\gamma$ is a contour (in the ``$x_i$-plane'') around $x_j$ containing no other singular points of $f$, i.e., none of the $x_k$ for $k \neq j$ nor the origin $0$. For definiteness we might take $\gamma$ to be $x_i = x_j + r e^{it}$ for $0 \leq t \leq 2\pi$ where
\begin{align*}
0 < r < R = \min \left( \{|x_j|\} \cup \{|x_j-x_k| \mid k \neq j\} \right).
\end{align*}
The Laurent series expansion \eqref{eq:i-expansion} for $f(x_i, x_j)$ is convergent in an open disc of radius $R$ centred on $x_j$. As the $x_k \neq x_j$ vary, the coefficients $f^{(m)}$ vary holomorphically, and the radius of convergence varies too.

The coefficient functions $f^{(m)}(x_1, \ldots x_{i-1}, x_{i+1} \ldots, x_n)$ lie in $\Gamma_{n-1}$. Indeed the value of $f^{(m)}$ is well defined for each collection $\{x_k\}$ of variables distinct from $0$ and from each other. The function is single-valued because the contour $\gamma$ can always be chosen to contain $x_j$ and no $x_k$ for $k \neq j$. It is also clear that $f^{(m)}$ is elliptic, because for any $\lambda \in \Lambda$ we have $f(w+\lambda, x_j+\lambda) = f(w, x_j)$. 
In summary
\begin{align}\label{eq:res.ser.exp}
\res_{x_i=x_j} f(x_i, x_j) Y(a^i, x_i-x_j) a^j = \sum_{m \geq N} f^{(m)}(x_j) a^i_{(m)}a^j,
\end{align}
where the functions $f^{(m)}$, defined by \eqref{eq:fm-integral}, are elements of $\Gamma_{n-1}$.

\end{nolabel}

\begin{prop}
For $d$ as defined above, $d((\partial_i + T_i) \widetilde{C}_n) = 0$.
\end{prop}

\begin{proof}
It is clear that $\partial_i + T_i$ commutes with $p^{(j)}$ for $j \neq i$ because the two operations apply to disjoint sets of variables and tensor factors. More precisely $p^{(j)} \circ (\partial_i + T_i) = (\partial_i + T_i) \circ p^{(j)}$ for $i < j$ while $p^{(j)} \circ (\partial_i + T_i) = (\partial_{i-1} + T_{i-1}) \circ p^{(j)}$ for $i > j$, due to the re-indexing of variables that occurs in the definition of $d$. Similar considerations apply to $\partial_i + T_i$ and $d^{(jk)}$ for $i \neq j$, $i \neq k$.

Next we confirm that $p^{(i)} \circ (\partial_i + T_i) = 0$, because of the calculation
\begin{align*}
\res_x f(x) Y(Ta, x)b = \res_x f(x) \partial_x Y(a, x)b = -\res_{x=0} (\partial_x f(x)) Y(a, x)b.
\end{align*}
Essentially the same calculation demonstrates that $d^{(ij)} \circ (\partial_i + T_i) = 0$. Finally we claim that $d^{(ij)} \circ (\partial_j + T_j) = (\partial_{j-1} + T_{j-1}) \circ d^{(ij)}$. For this we need to show that
\begin{align}\label{eq:d+T.sub.3}
\begin{split}
& \res_{x=y} f(x, y) Y(a, x-y) Tb + \res_{x=y} \left( \partial_y f(x, y) \right) Y(a, x-y) b \\
= & {} T \res_{x=y} f(x, y) Y(a, x-y) b + \partial_y \res_{x=y} f(x, y) Y(a, x-y) b.
\end{split}
\end{align}
The difference between the first terms on the left and right hand sides of \eqref{eq:d+T.sub.3} is
\begin{align*}
\res_{t=0} f(y+t, y) Y(Ta, t) b
&= \res_{t=0} f(y+t, y) \partial_t Y(a, t) b \\
&= -\res_{t=0} \partial_t f(y+t, y) Y(a, t) b.
\end{align*}
We can rewrite this as
\begin{align*}
- \partial_y \res_{t=0} f(y+t, y) Y(a, t) b + \res_{t=0} \partial_2 f(y+t, y) Y(a, t) b,
\end{align*}
where $\partial_2$ stands for the partial derivative relative to the second argument. We thus obtain the desired conclusion.
\end{proof}

\begin{prop} The differential $d$ defined by equations \eqref{eq:general.diff.def}-\eqref{eq:general.diff.def.comp2} in $C_\bullet$, satisfies $d \circ d = 0$.
\end{prop}

\begin{proof}
The basic pattern of the proof is the same as in Section \ref{sec:Leibniz}. First we confirm that the two terms analogous to \eqref{eq:crossing} cancel each other. Because of \eqref{eq:res.ser.exp} and \eqref{eq:fm-integral}, each term is a finite sum over $m, m' \in \Z$ of
\begin{align*}
a^1 \otimes \cdots \otimes a^i_{(m)}a^j \otimes \cdots  \otimes a^k_{(m')}a^\ell \otimes \cdots \otimes a^{n+2},
\end{align*}
with coefficient
\begin{align*}
\frac{1}{(2\pi i)^2} \oint_{\gamma} \oint_{\gamma'} \frac{f(x_i, x_j, x_k, x_\ell)}{(x_i-x_j)^{m+1} (x_k-x_\ell)^{m'+1}} \, dx_i \, dx_k,
\end{align*}
where $x_i = \gamma(t)$ is a small contour around $x_j$ and $x_k = \gamma'(t')$ is a small contour around $x_\ell$. The value of the coefficient is independent of the order of integration. With our choice of signs in \eqref{eq:general.diff.def}, we obtain cancellation just as in Section \ref{sec:Leibniz}.

It remains to treat the terms in which, for some $i < j < k$, both $a^i$ and $a^j$ are applied to $a^k$. These are (again suppressing spectator variables from the notation)
\begin{align*}
&\res_{x_i = x_k} Y(a^i, x_i-x_k) \left[ \res_{x_j = x_k}  f(x_i, x_j, x_k) Y(a^j, x_j-x_k) a^k \right] \\
&\res_{x_j = x_k} Y(a^j, x_j-x_k) \left[ \res_{x_i = x_k}  f(x_i, x_j, x_k) Y(a^i, x_i-x_k) a^k \right] \\
\text{and} \quad &\res_{x_j = x_k} Y( \left[ \res_{x_i = x_j}  f(x_i, x_j, x_k) Y(a^i, x_i-x_j) a^j \right], x_j-x_k).
\end{align*}
No essential change occurs if we remove some of the clutter by setting $x_k=0$. Thus we are considering
\begin{align*}
&\res_{x_i = 0} Y(a^i, x_i) \left[ \res_{x_j = 0}  f(x_i, x_j) Y(a^j, x_j) a^k \right] \\
&\res_{x_j = 0} Y(a^j, x_j) \left[ \res_{x_i = 0}  f(x_i, x_j) Y(a^i, x_i) a^k \right] \\
\text{and} \quad &\res_{x_j = 0} Y( \left[ \res_{x_i = x_j}  f(x_i, x_j) Y(a^i, x_i-x_j) a^j \right], x_j).
\end{align*}
It is now clear that these three terms are expressed by the three corresponding contour integrals in \eqref{eq:bor.contour} for sufficiently small $r > 0$. Once again the signs work out just as in Section \ref{sec:Leibniz}, and we have cancellation.
\end{proof}

\begin{rem}\label{rem:comparison.BD}
We now comment on the relationship between the constructions presented above and those of Beilinson and Drinfeld \cite{BD}. Let $\CA$ be a chiral algebra on an algebraic curve $X$. One obtains from $\CA$ a Lie algebra object $\g_\CA$ in a certain tensor category of $\CD$-modules on the Ran space $\Ran(X)$. This Lie algebra has a Chevalley-Eilenberg complex $\text{CE}(\g_\CA)$, which is in particular a complex of $\CD$-modules on $\Ran(X)$, and the chiral homology $\Hch_\bullet(X, \CA)$ is defined to be the global sections over $\Ran(X)$ of the de Rham homology of $\text{CE}(\g_\CA)$. The notion of a complex of sheaves on $\Ran(X)$ is formalised as a diagram of complexes of sheaves over the spaces $X^I$, as $I$ runs over the set of nonempty finite sets, with arrows induced by surjections $I \rightarrow I'$. A global section over $\Ran(X)$ of such an object is then defined to be a compatible collection of global sections over the individual $X^I$.

The complex $C_\bullet$ which we have defined above is the Beilinson-Drinfeld construction, done explicitly in the case in which $X = E$ is a smooth elliptic curve, but taking no account of higher derived functors of de Rham homology. For this reason, our complex $C_\bullet$ does not necessarily compute chiral homology in their sense in all degrees. However, as we have proved in {\cite{EH21}}, it does compute Beilinson-Drinfeld chiral homology in degrees $0$ and $1$.

In fact the passage to global sections described above does not quite correspond to $C_\bullet$ but rather to the quotient complex $Q_\bullet$ defined by $Q_n = (C_n)_{\Sigma_n}$, where the action of the symmetric group $\Sigma_n$ is as described in \S\nolinebreak \ref{nolabel:symm.action} below. We now verify directly that the differentials \eqref{eq:general.diff.def} are compatible with passage to coinvariants. This will use the skew-symmetry identity \eqref{eq:skew-symm}, which, we note, has not explicitly appeared in the definition of $C_\bullet$ nor in the verification of $d \circ d = 0$.
\end{rem}

\begin{nolabel}\label{nolabel:symm.action}
Now we introduce actions of the symmetric groups in the terms $C_n$ of our complex. To reduce clutter, here and below, we shall often shorten
\begin{align*}
f(x_1, \ldots, x_n) a^1 \otimes a^2 \otimes \ldots \otimes a^n \otimes b    
\end{align*}
to
\begin{align*}
f(x_1, \ldots, x_n) \cdot a^1 a^2 \ldots a^n \mid b.
\end{align*}
For $\sigma \in \Sigma_n$ we set
\begin{align}\label{eq:symm.grp.action}
\sigma\left( f(x_1, \ldots, x_n) \cdot a^1 \cdots a^n \mid b \right) = (-1)^{|\sigma|} f(x_{\sigma(1)}, \ldots, x_{\sigma(n)}) \cdot a^{\sigma(1)} \cdots a^{\sigma(n)} \mid b.
\end{align}
Let us denote by $I_n \subset C_n$ the subspace spanned by elements $A - \sigma(A)$ as $A$ runs over $C_n$ and $\sigma$ over $\Sigma_n$. Thus the spaces of coinvariants are given by $Q_n = (C_n)_{\Sigma_n} = C_n / I_n$. In the following proposition we verify that the coinvariants form a complex with the induced differential.
\end{nolabel}

\begin{prop}
The spaces of coinvariants $Q_n = (C_n)_{\Sigma_n}$, with differential induced from $C_\bullet$, form a complex.
\end{prop}

In the sequel we shall refer to $C_\bullet$ and $Q_\bullet$ as the (ordered and unordered, respectively) chiral complexes of $V$.

\begin{proof}
We wish to show that $d(I_n) \subset I_{n-1}$, and to do so it suffices to check $d(A - \tau(A)) \in I_{n-1}$ for a transposition $\tau = (\alpha, \alpha+1)$, where $1 \leq \alpha \leq n-1$, and where
\begin{align*}
A = f(x_1, \ldots, x_\alpha, \ldots x_n) a^1 \cdots a^\alpha \cdots a^n \mid b.
\end{align*}
From \eqref{eq:general.diff.def.comp1} we have
\begin{align*}
d^{(i \alpha)}(A) = \res_{x_i=x_\alpha} f(x_i, x_\alpha, x_{\alpha+1}) (Y(a^i, x_i-x_\alpha)a^\alpha) a^{\alpha+1} \mid b
\end{align*}
(suppressing spectator variables from the notation) and similarly
\begin{align*}
d^{(i \, \alpha+1)}(\tau(A)) = \res_{x_i=x_{\alpha+1}} f(x_i, x_{\alpha+1}, x_{\alpha})  a^{\alpha+1} (Y(a^i, x_i-x_{\alpha+1})a^\alpha) \mid b.
\end{align*}
Up to a sign, these two terms differ by simultaneous exchange of $a^\alpha$ with $a^{\alpha+1}$ and $x_\alpha$ with $x_{\alpha+1}$. That is, up to a sign, these two terms are related by the action of a transposition in $\Sigma_{n-1}$. Taking account of the signs imposed in \eqref{eq:general.diff.def}, we verify that the sum of these contributions to $d(A - \tau(A))$ lies in $I_{n-1}$. Similar arguments apply to the pairs of terms
\begin{align*}
d^{(i \, \alpha+1)}(A) \quad \text{and} \quad d^{(i \, \alpha)}(\tau(A)) \quad &\text{for $i < \alpha$}, \\
d^{(\alpha \, j)}(A) \quad \text{and} \quad d^{(\alpha+1 \, j)}(\tau(A)) \quad &\text{for $j > \alpha+1$}, \\
d^{(\alpha+1 \, j)}(A) \quad \text{and} \quad d^{(\alpha \, j)}(\tau(A)) \quad &\text{for $j > \alpha+1$}, \\
p^{(\alpha)}(A) \quad \text{and} \quad p^{(\alpha+1)}(\tau(A)) \quad & \\
p^{(\alpha+1)}(A) \quad \text{and} \quad p^{(\alpha)}(\tau(A)) \quad &
\end{align*}
In case the sets $\{i, j\}$ and $\{\alpha, \alpha+1\}$ are disjoint, it is clear that $d^{(ij)}(A-\tau(A))$ and $p^{(i)}(A-\tau(A))$ lie in $I_{n-1}$.

It remains to verify that $d^{(\alpha \, \alpha+1)}(A - \tau(A)) \in I_{n-1}$, and here we use the skew-symmetry identity \eqref{eq:skew-symm} in the vertex algebra $V$. To improve readability we write $x := x_\alpha$, $y := x_{\alpha+1}$, and $a := a^{\alpha}$ and $a' := a^{\alpha+1}$, so that it is required to show
\begin{align*}
-\res_{x=y} f(x, y) Y(a, x-y)a' \equiv \res_{x=y} f(y, x) Y(a', x-y) a \pmod{(T+\partial_y)}.
\end{align*}
Setting $t = x-y$, the left hand side becomes $-\res_{t=0} f(y+t, y) Y(a, t)a'$, and the right hand side becomes
\begin{align*}
& \res_{t=0} f(y, y+t) Y(a', t)a
= \res_{t=0} f(y, y+t) e^{tT} Y(a, -t)a' \\
\equiv& \res_{t=0} f(y, y+t) e^{-t \partial_{y}} Y(a, -t)a'
= \res_{t=0} f(y-t, y) Y(a, -t)a' \\
=& -\res_{t=0} f(y+t, y) Y(a, t)a'.
\end{align*}
So we obtain the desired equivalence.
\end{proof}

\section{Filtration on the chiral complex: Motivation}

\begin{nolabel}
To motivate the general constructions to be introduced below, we examine homology of $C_\bullet$ in degree $0$. We have $C_0 = V$, and $C_1$ a certain quotient of $\Gamma_1 \otimes V \otimes V$. Indeed $\Gamma_1$ is linearly spanned by the functions $1$, $\wp(x_1)$ and the derivatives $\partial_1^k \wp(x_1)$ for $k \geq 1$, and the quotient is by the action of $\partial_{1} + T_1$. For $k \geq 1$ we see that $\partial_1^k \wp(x_1) \cdot a^1 \mid b$ is identified with $-\partial_1^{k-1} \wp(x_1) \cdot (Ta^1) \mid b$, so that in fact $C_1$ is spanned by two kinds of terms: $\wp(x_1) \cdot a^1 \mid b$ and $1 \cdot a^1 \mid b$. We recover the familiar description of coinvariants
\begin{align*}
H_0(C) = V / \text{Im}(d), \quad \text{Im}(d) = \left< a_{(\wp)}b, a_{(0)}b \mid a, b \in V \right>. 
\end{align*}
Here $a_{(\wp)}b$ is shorthand for
\begin{align*}
\res_x \wp(x) Y(a, x) b = a_{(-2)}b + \sum_{k = 1}^\infty (2k+1) G_{2k+2}(\tau) a_{(2k)}b.
\end{align*}
\end{nolabel}

\begin{nolabel}
Let us now take a good increasing filtration $G$ on $V$, for example a standard filtration. We may equip $\Gamma_1 \otimes V \otimes V$ with the induced filtration on $V \otimes V$, extending scalars from $\C$ to $\Gamma_1$. Since $T$ has degree $0$ relative to $G$, this yields a filtration on the quotient $C_1$. The differential $d : C_1 \rightarrow C_0$ is compatible with the filtration. Let us denote $A = \grg{}(V)$. Upon passage to the associated graded, the images of $\wp(x_1) \cdot a^1 \mid b$ and $1 \cdot a^1 \mid b$ in $A$ are, respectively, $a_{(-2)}b$ and $0$. Thus by standard theory of filtrations $\dim R_V < \infty$ implies $\dim H_0(C) < \infty$.
\end{nolabel}

\begin{nolabel}
The implication just stated was proved by Zhu {\cite[Lemma 4.4.1]{zhu96}} with an argument by induction on conformal weight (the theory of filtrations on vertex algebras had not yet been developed at the time). The same proof implies finite dimensionality of $\dim H_0(C)$ subject to the weaker hypothesis $\dim \HP_0(R_V) < \infty$, since $\HP_0(R_V) = R_V / \{R_V, R_V\} = V / \left<V_{(-2)}V, V_{(0)}V \right>$ {\cite[Proposition 5.2]{Arakawa.Kawasetsu}}.
\end{nolabel}

\begin{nolabel}
Motivated by these observations, and desiring to establish a link with Poisson homology of $R_V$ at higher degrees, we seek a filtration on $C_\bullet$ in whose associated graded both terms $a_{(-2)}b$ and $a_{(0)}b$ appear in the image of $d : \gr(C_1) \rightarrow \gr(C_0)$. This is achieved by working with nontrivial filtrations on the spaces $\Gamma_n$. Ultimately the right filtration for our needs comes from Hodge and weight filtrations on the de Rham cohomology of configuration spaces of the elliptic curve $E$.
\end{nolabel}

\begin{nolabel}\label{nolabel:explicit.C.Hodge}
We now describe a filtration $G$ on $\Gamma_n$ which will form part of the definition of the Hodge filtration on $C_\bullet$ to be given in Section \ref{sec:first.filtration} below. We set $\filg{-n+p}\Gamma_n \subset \Gamma_n$ to be the vector subspace spanned by functions which are, locally, the product of at most $p$ functions with a pole on a single one of the divisors $x_i=0$ and $x_i=x_j$ in $X^n$. This is an increasing filtration on $\Gamma_n$ with $\filg{-n-1}\Gamma_n = 0$ and $\filg{-n}\Gamma_n = \C \cdot 1$. Clearly this filtration is stable under the action of the derivations $\partial_i$ on $\Gamma_n$.
\end{nolabel}

\begin{nolabel}\label{nolabel:deg0.HP}
For example $G_{-1}\Gamma_1$ is spanned by the function $1$ and $G_0\Gamma_1 = \Gamma_1$. This filtration, together with a choice of good increasing filtration on $V$, induces a filtration on $C_1$, which we also refer to as $G$. If $a^1 \in \filg{p}V$ and $b \in \filg{q}V$ then $\wp(x_1) \cdot a^1 \mid b \in \grg{p+q}C_1$ and $d(\wp(x_1) \cdot a^1 \mid b) = a^1_{(\wp)}b \equiv a^1_{(-2)}b$ in the associated graded space $\grg{p+q}V$. If $a^1 \in \filg{p+1}V$ then $1 \cdot a^1 \mid b \in \grg{p+q}C_1$ also, and $d(1 \cdot a^1 \mid b) = a^1_{(0)}b$ in $\grg{p+q}V$. Thus we recover $H_0(C_\bullet)$ via a spectral sequence with $\HP_0(R_V)$ on the first page, and in particular the dimension of the former is bounded above by the dimension of the latter.
\end{nolabel}

\begin{rem}
We will justify the name ``Hodge filtration'' in Section \ref{sec:Hodge.geometry} below, where Hodge and weight filtrations in cohomology are discussed, but the reader may safely skip directly to Section \ref{sec:first.filtration} if they prefer. In particular, we will recall the role played by complexes of logarithmic differential forms in the construction of the usual Hodge and weight filtrations on the cohomology of algebraic varieties. In fact a variant on the chiral chain complex incorporating differential forms with logarithmic poles on diagonal divisors has been briefly discussed in {\cite[Section 4.2.13]{BD}}.
\end{rem}

\begin{rem}
In the series of papers \cite{BDHK19, BDK20, BDKV21, BDK21, BDHKV21} a local theory of the chiral operad of \cite{BD} has been developed and studied. Filtrations similar to the one described in \S\nopagebreak \ref{nolabel:explicit.C.Hodge} above are considered, and structures closely related to the variational complex of \cite{DSK.variational} are recovered upon passage to associated graded complexes. It would be interesting to investigate connections between these works and the approach developed in the present article.
\end{rem}

\section{Cohomology of configuration spaces of elliptic curves}\label{sec:Hodge.geometry}

\begin{nolabel}
Let $E$ be our elliptic curve $\C / \Lambda$ and $\ovX = E \backslash \{0\}$, as above. In this section we describe the coordinate rings $\Gamma_n$ and the cohomology of the affine varieties $F(\ovX, n)$ in more detail.
\end{nolabel}

\begin{nolabel}
The ring $\Gamma_n$ contains the pullbacks of the Weierstrass $\wp$-function $\wp(x_i)$ for $1 \leq i \leq n$, and products thereof. Now recall the Weierstrass $\zeta$-function, characterised by its Laurent expansion
\begin{align*}
\zeta(x) = x^{-1} - \sum_{k = 0}^\infty G_{2k+2}(\tau) x^{2k+1},
\end{align*}
which is meromorphic on $\C$ with simple poles on the points of $\Lambda$, and satisfies
\begin{align}\label{eq:zeta.quasi}
\zeta(x+1) = \zeta(x), \qquad \zeta(x+\tau) = \zeta(x) - 2\pi i.
\end{align}
We have $\wp(x) = -\zeta'(x) - G_2(\tau)$. While $\zeta(x)$ is not elliptic, it is immediate to verify that
\begin{align*}
\zeta(x, y) = \zeta(x-y) - \zeta(x) + \zeta(y)
\end{align*}
is a meromorphic function on $E^2$ with poles at $x=0$, $y=0$ and $x=y$. Thus, for $n \geq 2$, the ring $\Gamma_n$ contains the functions $\zeta(x_i, x_j)$ for $1 \leq i < j \leq n$, as well as their products with each other and with the $\wp(x_i)$, and derivatives thereof. The functions $\wp(x_i)$ and $\zeta(x_i, x_j)$ do not generate $\Gamma_n$ as an algebra in general (indeed not even for $n=1$), however we have the following result.
\end{nolabel}

\begin{prop}[{\cite[Lemma 9.3]{EH21}}]\label{prop:Gamma.gens}
Denote by $H_n \subset \Gamma_n$ the subalgebra generated by $\wp(x_i)$ for $1 \leq i \leq n$ and $\zeta(x_i, x_j)$ for $1 \leq i < j \leq n$. Then the quotient vector space $\Gamma_n / \left<\partial_i \Gamma_n\right>_{i=1,\ldots, n}$ is linearly spanned by the image of $H_n$.
\end{prop}

\begin{nolabel}
Returning to the filtration defined in \S\nolinebreak \ref{nolabel:explicit.C.Hodge}, we have $\wp(x_i), \zeta(x_i, x_j) \in \filg{-n+1} \Gamma_n$, and the product of $p$ of these generators lies in $\filg{-n+p}\Gamma_n$. Recall that $C_n$ is a quotient of $\Gamma_n \otimes V^{\otimes n+1}$, in which a derivative $\partial_i$ acting on $f \cdot a^1 a^2 \cdots a^n \mid b$ may be traded for a copy of $T$ acting on $a^i$. Proposition \ref{prop:Gamma.gens} then implies that any element of $C_n$ has a representative in the subspace $H_n \otimes V^{\otimes n+1}$.
\end{nolabel}

\begin{rem}
In terms of the algebraic coordinates $u$ and $v$ (see \S\nolinebreak \ref{nolabel:algebraic.coords.E}) the function $\zeta(x_i, x_j)$ is expressed as $(v_i+v_j)/(u_i-u_j)$, due to the following classical elliptic function identity \cite[(18.4.3)]{Ab.Steg}:
\begin{align}\label{eq:zeta.wp.identity}
\zeta(x-y) - \zeta(x) + \zeta(y) = \frac{1}{2} \frac{\wp'(x)-\wp'(y)}{\wp(x)-\wp(y)}.
\end{align}    
\end{rem}

\begin{nolabel}\label{nolabel:dR.sing}
Let $Y$ be a smooth complex algebraic variety of dimension $n$. For a right $\CD_Y$-module $M$ a sheaf of vector spaces $h(M)$ on $Y$ is defined by
\begin{align*}
h(M) = M / (M \cdot \Theta_Y),
\end{align*}
where $\Theta_Y \subset \CD_Y$ is the tangent bundle. As a functor between derived categories, the de Rham homology of right $\CD$-modules is the composition of $h$ with the global sections functor $\Gamma$. The de Rham homology of the canonical bundle $\omega_Y = \wedge^n \Omega_Y$ (which is naturally a right $\CD_Y$-module) recovers the singular cohomology of $Y$ up to a shift. In particular if we take $Y$ affine then $\Gamma$ is exact, and we recover
\begin{align*}
\Gamma(Y, h(M)) \cong H^{n}(Y).
\end{align*}
Restricting attention to $Y = F(\ovX, n)$, for which $\omega_Y \cong \OO_Y$, we arrive at the identification
\begin{align}\label{eq:dR.sing}
\Gamma_n / \left<\partial_i \Gamma_n\right>_{i=1,\ldots, n} \cong H^{n}(F(\ovX, n)).
\end{align}
\end{nolabel}

\begin{nolabel}
Since $F(\ovX, n)$ is an affine algebraic variety, its cohomology can be computed as the cohomology of the algebraic de Rham complex $\Omega_{F(\ovX, n)}^\bullet$, which is the exterior algebra of the module of K\"{a}hler differentials of $\Gamma_n$.
\end{nolabel}


\begin{nolabel}
We now discuss the Hodge and weight filtrations on the cohomology $H^*(Y)$ of a smooth complex algebraic variety $Y$ of dimension $n$. If $Y$ is compact then the Hodge filtration $H^*(Y) = \filf{0} \supset \filf{1} \supset \ldots$ can be described as follows. Cohomology classes are represented by Dolbeault forms of type $(k, \ell)$, which is to say differential forms locally of the form $f \, dx_{i_1} \cdots dx_{i_k} d\ov{x}_{j_1} \cdots d\ov{x}_{j_\ell}$ where $f$ is a $C^\infty$ function. The Hodge filtration (which for $Y$ compact, though not in general, is actually induced by a $\Z$-grading) is obtained by setting $F^p$ to be that part of the cohomology spanned by the classes of Dolbeault forms of type $(k, \ell)$ for $k \geq p$.
\end{nolabel}

\begin{nolabel}
The extension of the definition of Hodge filtration to noncompact $Y$ has been given by Deligne \cite{Deligne}: First embed $Y$ into a smooth projective variety $\ovy$ as the complement of a normal crossings divisor $D$. That is $Y = \ovy \backslash D$. Suppose that $D \subset \ovy$ is defined locally as the vanishing locus of a function $f$. One considers the complex of sheaves of differential forms on $\ovy$ with logarithmic singularities on $D$, denoted $\Omega_{\ovy}^\bullet(\log D)$. By definition a differential form $\omega$ on ${\ovy}$ has logarithmic singularities on $D$ if locally $f \omega$ and $f d\omega$ are holomorphic \cite{Saito}. The Hodge filtration is obtained applying the Hodge-de Rham spectral sequence to $\Omega_{\ovy}^\bullet(\log D)$. We explain the details briefly. In general, for a cohomological complex $A^\bullet$, the truncation $\sigma_{\geq p} A^\bullet$ of $A^\bullet$ is the complex defined by
\begin{align*}
(\sigma_{\geq p} A)^i = \begin{dcases*}
A^i & if $i \geq p$, \\
0 & if $i < p$. \\
\end{dcases*}
\end{align*}
The Hodge-de Rham spectral sequence is a spectral sequence whose first page is
\begin{align*}
E_1^{p, q} = H^q({\ovy}, \Omega_{\ovy}^p(\log D)),
\end{align*}
and which converges to the cohomology $H^*(Y)$ of $Y$. The decreasing filtration consisting of truncated subcomplexes $\sigma_{\geq p} \Omega_{\ovy}^\bullet(\log D)$ induces a decreasing filtration on $E_1^{\bullet, \bullet}$, and the Hodge filtration on $H^*(Y)$ is then defined to be the induced filtration \cite[{Tag 0FM7}]{stacks-project}
\begin{align*}
\filf{p}H^n(Y) = \text{Im}(H^n({\ovy}, \sigma_{\geq p} \Omega_{\ovy}^\bullet(\log D)) \rightarrow H^n(Y)).
\end{align*}
\end{nolabel}

\begin{nolabel}
Let us illustrate the definition by working out the simple example of $Y = \ovX = {E} \backslash \{0\}$ embedded in ${\ovy} = E$. From the algebraic de Rham complex of $\ovX$ we obtain $\dim H^0(\ovX) = 1$ with basis given by the function $1$ and $\dim H^1(\ovX) = 2$ with basis given by the $1$-forms $dx$ and $\wp(x)\, dx$. All higher cohomology groups vanish. On the other hand we have $\Omega^0_{\ovy}(\log{0}) = \Omega^0_{\ovy} = \OO_{\ovy}$, so that $E_1^{0, 0} = E_1^{0, 1} = \C$, and also $\Gamma(\Omega^1_{\ovy}(\log{0})) = \Gamma(\Omega^1_{\ovy}) = \C \, dx$ since there are no meromorphic forms with pole of order $1$. Thus the Hodge filtration on $H^0(\ovX)$ is given by $\filf{0} = H^0$, and the Hodge filtration on $H^1(\ovX)$ is given by $\filf{0} = H^1$ and $\filf{1} = \C \, [dx]$. We can also compare this with the $C^\infty$ picture. Due to \eqref{eq:zeta.quasi} the $C^\infty$ function (actually real analytic function)
\begin{align*}
f(x, \ov{x}) = \zeta(x) + 2\pi i \frac{\ov{x}-x}{\ov{\tau}-\tau}
\end{align*}
is $\Lambda$-periodic on $\C$, hence descends to a $C^\infty$ function on $\ovX$. In the $C^\infty$ de Rham complex of $\ovX$ we therefore have
\begin{align}\label{eq:wp.is.dxbar}
d\ov x = \frac{\ov\tau - \tau}{2\pi i} \left( \wp(x) + G_2(\tau) \right) dx + dx.
\end{align}
We have, once again, that $[dx]$ spans $\filf{1}$, and in the quotient $\filf{0} / \filf{1}$ we have the relation $[\wp(x) \, dx] = \frac{2\pi i}{\ov\tau - \tau} [d\ov{x}]$.
\end{nolabel}

\begin{nolabel}
The cohomology of a smooth, not necessarily compact, complex algebraic variety famously comes equipped with a mixed Hodge structure \cite{Deligne}. Such a structure consists of two filtrations: the decreasing Hodge filtration described above, and an increasing filtration called the weight filtration. The weight filtration $\filw{0} \subset \ldots \subset \filw{n} = H^*(Y)$ for $Y = {\ovy} \backslash D$ as above, is induced by the increasing filtration of $\Omega_{\ovy}^\bullet(\log{D})$ by subcomplexes $\filw{\ell}$, where $\filw{\ell}$ is defined to be the span of regular functions times forms locally expressible as $\tfrac{1}{X_I} dx_I \wedge dx_J$ for $|I| \leq \ell$. {\cite[p. 208]{Voisin-vol1}}. 
\end{nolabel}

\begin{nolabel}
In what follows we shall work with the obvious embedding $F(X, n) \subset E^n$, whose complement $D$ is the union of divisors $\{x_i = 0\}$ and $\{x_i = x_j\}$, for $1 \leq i, j \leq n$ and $i \neq j$. Despite the fact that $D$ is not normal crossings, we shall work with the filtrations $F$ and $W$ described above, and relate them to the filtration $G$ on $\Gamma_n$, defined in \S\nolinebreak \ref{nolabel:explicit.C.Hodge}.
\end{nolabel}

\begin{nolabel}\label{nolabel:compat.of.G}
Let $\filf{p} = \filf{p}H^k(F(\ovX, n))$ and $\filw{p} = \filw{p}H^k(F(\ovX, n))$. We set
\begin{align}\label{eq:G.filtration.def.coho}
\filg{p}H^k(F(\ovX, n)) = \sum_{p' + p'' \leq p} \filf{-p'} \cap \filw{p''+n}.
\end{align}
This is an increasing filtration on $H^k(F(\ovX, n))$.
\end{nolabel}

\begin{prop}\label{prop:filtr.compat}
The filtration $G$ on $\Gamma_n$ which was defined in \S\nolinebreak \ref{nolabel:explicit.C.Hodge} and the filtration $G$ on $H^k(F(\ovX, n))$ defined by equation \ref{eq:G.filtration.def.coho} are compatible with the surjection $\Gamma_n \rightarrow H^k(F(\ovX, n))$ given by equation \eqref{eq:dR.sing}.
\end{prop}

\begin{nolabel}
The rest of this section is devoted to the proof of Proposition \ref{prop:filtr.compat}. In order to effect the comparison between the filtrations, we use a description of the cohomology of $F(\ovX, n)$ given by Totaro \cite{Totaro}. His results are in the general context of an $n$-dimensional orientable real manifold $M$, but we only need the case $M = \ovX = E \backslash \{0\}$.
\end{nolabel}

\begin{nolabel}
Let $G(n)$ denote the graded commutative algebra with generators $G_{ij}$, for $1 \leq i, j \leq n$ and $i \neq j$, all of degree $1$ and relations
\begin{align}
G_{ij} = G_{ji}, \quad G_{ij}^2 &= 0, \quad \text{for $i, j$ distinct}, \label{G.symmetry} \\ 
G_{ij}G_{ik} + G_{jk}G_{ji} + G_{ki}G_{kj} &= 0, \quad \text{for $i, j, k$ pairwise distinct}. \label{eq:G.Gelfand}
\end{align}
\end{nolabel}

\begin{nolabel}
Now let $\pi_i : \ovX^n \rightarrow \ovX$ denote the projection to the $i^{\text{th}}$ component, and similarly $\pi_{ij} : \ovX^n \rightarrow \ovX^2$. Denote by $\Delta \in H^2(\ovX^2)$ the class of the diagonal (see \S\nolinebreak \ref{nolabel:diag.class} below for an explicit description of $\Delta$). Now denote by $E_2(n)$ the bigraded algebra
\begin{align*}
H^*(\ovX)^{\otimes n} \otimes G(n) / \left< (\pi_i^*\eta - \pi_j^*\eta) \otimes G_{ij} \mid \eta \in H^*(\ovX) \right>,
\end{align*}
with $H^k(\ovX)$ in bidegree $(k, 0)$ and all $G_{ij}$ in bidegree $(0, 1)$, and differential $d$ of bidegree $(2, -1)$ defined by
\begin{align}\label{eq:Gysin}
d(G_{ij}) = \pi_{ij}^*\Delta.
\end{align}
The following is a special case of {\cite[Theorem 1]{Totaro}}.
\end{nolabel}

\begin{thm}
There is a spectral sequence converging to $H^*(F(\ovX, n))$ whose $E_2$ page is the bigraded differential algebra $E_2(n)$.
\end{thm}

\begin{nolabel}\label{nolabel:diag.class}
We can write down $\Delta \in H^2(\ovX^2)$ explicitly as follows. For closed orientable manifolds $M$ \cite[Theorem 11.11]{Milnor-Stasheff} the class of the diagonal is given by
\begin{align*}
\sum_i (-1)^{\deg(e_i)} \pi_1^*(e_i) \pi_2^*(e^i)
\end{align*}
where the sum runs over a homogeneous basis $\{e_i\}$ of $H^*(M)$, and $\{e^i\}$ is the basis dual with respect to the Poincar\'{e} pairing. For the elliptic curve $E$, the Poincar\'{e} pairing on $H^*(E)$ can be deduced from the integral
\begin{align*}
\int_{[E]} \, dx \wedge d\ov x = \ov \tau - \tau,
\end{align*}
from which it quickly follows that
\begin{align}\label{eq:diag.class}
\Delta = \frac{1}{\ov\tau - \tau} \left( dx_1 - dx_2 \right) \wedge \left( d\ov x_1 - d\ov x_2 \right).    
\end{align}
For orientable but noncompact manifolds such as $M = \ovX$, Poincar\'{e} duality involves Borel-Moore homology, and $H^k(M) \cong H_{n-k}^{\text{BM}}(M)$, essentially since both these spaces have perfect pairings with $H^{n-k}_c(M)$. Ultimately the calculation to determine $\Delta$ is the same for $\ovX$ as that above for $E$ (pairing against compactly supported differential forms instead of all forms), and the class $\Delta$ is thus given by the same formula \eqref{eq:diag.class}.
\end{nolabel}

\begin{nolabel}
As explained in \cite{Totaro}, the differential \eqref{eq:Gysin} is compatible with mixed Hodge structures, since it is an example of a Gysin morphism. In the present case, the class $\pi_{ij}^* \Delta$ lies in Hodge degree $1$, and therefore so does $G_{ij}$.
\end{nolabel}

\begin{nolabel}
We illustrate with the case $n=2$. Let us write $H$ as shorthand for $H^*(\ovX)$. The cohomology $H^k(F(\ovX, 2))$ is nonzero for $k=0, 1$ and $2$ with dimension $1$, $4$ and $5$, respectively. The structure of $E_2(2)$ is as follows
\begin{align}\label{eq:spec.seq}
\begin{split}
\xymatrix{
1 & G_{12} H^0 \ar@{->}[drr]^{d} & G_{12} H^1 & \\
0 & (H^{\otimes 2})^0 & (H^{\otimes 2})^1 & (H^{\otimes 2})^2 \\
 & 0 & 1 & 2 \\
}    
\end{split}
\end{align}
where the Gysin map $d$ is the only nontrivial differential. We have seen above that the Hodge filtration in $H^1$ is $\filf{0}H^1 = H^1$ and $\filf{1}H^1$ spanned by the class $[dx]$. The Hodge filtration in $H^0$ is simply $\filf{0}H^0 = H^0$. The weight filtration in $H$ (and in $H^{\otimes n}$ in general) is the cohomological gradation $\filw{k} = H^k$, and the weight of $G_{ij}$ is $2$ {\cite[p. 1064]{Totaro}}. The cohomology $H^2(F(\ovX, 2))$ may have nontrivial contributions from positions $(2, 0)$ and $(1, 1)$ in Figure \eqref{eq:spec.seq}, and is $5$-dimensional, spanned by classes represented by
\begin{align*}
dx_1 \wedge dx_2, \quad d\ov{x}_1 \wedge d{x}_2, \quad d\ov{x}_1 \wedge d\ov{x}_2, \quad G_{12} \, dx \quad \text{and} \quad G_{12} \, d\ov{x}.
\end{align*}
These representatives lie, respectively, in Hodge degrees $h = 2, 1, 0, 2, 1$, and in weights $w = 2, 2, 2, 3, 3$. Their positions in the filtration $G$ defined in \eqref{eq:G.filtration.def.coho} are therefore $-h + w-2 = -2, -1, 0, -1, 0$, respectively. We have seen in equation \eqref{eq:wp.is.dxbar} above that $[d\ov{x}]$ is identified with $[\wp(x) \, dx]$ up to a scalar factor, and we will see below that $G_{ij}$ may be identified with $[\zeta(x_i, x_j) \, (dx_i - dx_j)]$ up to a scalar. We see then that the degrees just assigned to the four generators of $H^2(F(\ovX, 2))$ match with the definition given in \S\nolinebreak \ref{nolabel:explicit.C.Hodge} of the last section.
\end{nolabel}

\begin{nolabel}
To compare the algebraic de Rham complex $\Omega_{\dR}(\Gamma_n) = \wedge^\bullet \Omega_{\Gamma_n}$ of $\Gamma_n$ with Totaro's algebra $E_2(n)$, which are both commutative dg-algebras,  we pass via a third commutative dg-algebra with filtration very similar to one introduced by Brown and Levin in \cite{BL}. As in that paper, this algebra $\X_n$, defined below, admits dg-morphisms to $\Omega_{\dR}(\Gamma_n)$ and $E_2(n)$:
\begin{align*}
\xymatrix{
 & \X_n \ar@{->}[dr]^{\psi} \ar@{->}[dl]_{\varphi} & \\
\Omega_{\dR}(\Gamma_n) &  & E_2(n) \\
}
\end{align*}
compatible with filtrations (see Propositions \ref{prop:to.Omega} and \ref{prop:to.E2} and \S \ref{nolabel:filtr.Xi.Om.E2} below), thus proving Proposition \ref{prop:filtr.compat}.
\end{nolabel}

\begin{defn}
We denote by $\X_n$ the commutative dg-algebra with generators $\omega_i$ and $\nu_i$ for $i = 1, \ldots, n$ and $\omega_{ij}^{m}$ for $1 \leq i, j \leq n$ and $m \geq 0$ (all generators of degree $1$), satisfying the relations
\begin{align*}
\omega_{ij}^{(m)} + (-1)^{m} \omega_{ji}^{(m)} = 0, \qquad
\omega_{ij}^{(0)} = \omega_i - \omega_j,
\end{align*}
and the infinitely many quadratic relations given by the coefficients in powers of $\alpha$ and $\beta$ of
\begin{align}\label{eq:Fay.rel.X}
\Omega_{i\ell}(\alpha) \wedge \Omega_{j\ell}(\beta) + \Omega_{ji}(\beta) \wedge \Omega_{i\ell}(\alpha+\beta) + \Omega_{j\ell}(\alpha+\beta) \wedge \Omega_{ij}(\alpha) = 0,
\end{align}
where
\begin{align}\label{eq:Omegaij.exp}
\Omega_{ij}(\alpha) = \sum_{m \geq 0} \omega_{ij}^{(m)} \alpha^{m-1}.
\end{align}
The differential of $\X_n$ is defined by
\begin{align*}
d(\omega_i) = d(\nu_i) &= 0 \\
d(\omega_{ij}^{(m+1)}) &= (\nu_i - \nu_j) \wedge \omega_{ij}^{(m)}, \quad m \geq 0.
\end{align*}
\end{defn}


\begin{nolabel}
In order to write down the morphism $\X_n \rightarrow \Om_{\dR}(\Gamma_n)$ we need to introduce some more terminology. We write
\begin{align*}
\E_m(x) = \sum_{\lambda \in \Lambda \backslash 0} \left( \frac{1}{(x+\la)^m} - \frac{1}{\la^m} \right)
\end{align*}
(note that $\E_1 = \zeta$ and $\E_2 = \wp$, and for $k \geq 3$ the function $\E_k$ is proportional to a derivative of $\wp$) and
\begin{align*}
\wtil\Omega_{ij}(\alpha)
&= \alpha^{-1} \exp\left[\alpha \left( \E_1(x_i-x_j) - \E_1(x_i) + \E_1(x_j) \right) \right] \\
& \quad \quad \quad \quad \quad \quad \times \exp\left[ - \sum_{m \geq 2} \frac{(-\alpha)^m}{m} \E_m(x_i-x_j) \right] (dx_i-dx_j),    
\end{align*}
which we then expand as
\begin{align*}
\wtil\Omega_{ij}(\alpha) = \sum_{m \geq 0} \wtil\omega_{ij}^{(m)} \alpha^{m-1}
\end{align*}
in parallel with \eqref{eq:Omegaij.exp}. In terms of the functions $\wp(x_i)$ and $\zeta(x_i, x_j)$ we obtain
\begin{align*}
\wtil\om_{ij}^{(0)} &= dx_i-dx_j, \\
\wtil\om_{i,j}^{(1)} &= \zeta(x_i, x_j) (dx_i-dx_j), \\
\wtil\om_{ij}^{(2)} &= \frac{1}{2}\left[ \zeta(x_i, x_j)^2 - \wp(x_i-x_j) \right] (dx_i-dx_j),
\end{align*}
and so on. Now we can write down the morphism $\X_n \rightarrow \Om_{\dR}(\Gamma_n)$.
\end{nolabel}

\begin{prop}\label{prop:to.Omega}
The assignments
\begin{align*}
\varphi(\nu_i) &= \wtil\nu_i = \wp(x_i) dx_i, \\
\varphi(\om_i) &= \wtil\om_i = dx_i, \\
\varphi(\om_{ij}^{(m)}) &= \wtil\om_{ij}^{(m)},
\end{align*}
extend to a morphism of dg-algebras $\varphi : \X_n \rightarrow \Om_{\dR}(\Gamma_n)$.
\end{prop}

\begin{proof}
The relations \eqref{eq:Fay.rel.X}, with $\om$ everywhere replaced with $\wtil\om$, need to be verified. They hold because of the Fay trisecant identity, recalled as Theorem \ref{thm:Fay} below. The relations
\begin{align*}
\wtil{\omega}_{ij}^{(m)} + (-1)^{m} \wtil{\omega}_{ji}^{(m)} = 0 \quad \text{and} \quad
\wtil{\omega}_{ij}^{(0)} = \wtil{\omega}_i - \wtil{\omega}_j
\end{align*}
are evident from the symmetry properties of $\wtil{\Om}$. Also evident are the differential relations
\begin{align*}
d(\wtil\omega_i) = d(\wtil\nu_i) = 0.
\end{align*}
It remains to verify the relation
\begin{align*}
d(\wtil\omega_{ij}^{(m+1)}) = (\wtil\nu_i - \wtil\nu_j) \wedge \wtil\omega_{ij}^{(m)}, \quad m \geq 0.
\end{align*}
Most of the summands in the definition of $\wtil\Om_{ij}$ depend on $x_i$ and $x_j$ only through $x_i-x_j$. Schematically we might express this as
\begin{align*}
\wtil\Omega_{ij}(\alpha) = \exp(\alpha h + f) (dx_i-dx_j),
\end{align*}
where $h = -\zeta(x_i)+\zeta(x_j)$ and $f = f(x_i-x_j)$ is a function of $x_i-x_j$. It follows that
\begin{align*}
d \wtil\Omega_{ij}(\alpha)
= \alpha \, dh \wedge \wtil\Omega_{ij}(\alpha)
= \alpha (\wp(x_i) dx_i -\wp(x_j) dx_j ) \wedge \wtil\Omega_{ij}(\alpha),
\end{align*}
and so
\begin{align*}
\sum_{m \geq 0} d \wtil\omega_{ij}^{(m)} \alpha^{m-1}
&= (\wp(x_i) dx_i -\wp(x_j) dx_j ) \wedge \sum_{m \geq 0} \wtil\omega_{ij}^{(m)} \alpha^{m}.
\end{align*}
Therefore $d(\wtil\om_{ij}^{(0)}) = 0$ and 
\begin{align*}
d(\wtil\omega_{ij}^{(m+1)}) &= (\wtil\nu_i - \wtil\nu_j) \wedge \wtil\omega_{ij}^{(m)}, \quad m \geq 0.
\end{align*}
\end{proof}

\begin{thm}[The Fay trisecant identity]\label{thm:Fay}
Let
\begin{align*}
F(z, \alpha) = \alpha^{-1} \exp\left[ - \sum_{m \geq 1} \frac{(-\alpha)^m}{m} \E_m(z) \right],
\end{align*}
then
\begin{align}\label{eq:Fay}
F(z, \alpha) F(w, \beta) - F(z, \alpha+\beta) F(w-z, \beta) - F(w, \alpha+\beta) F(z-w, \alpha) = 0.
\end{align}
\end{thm}

\begin{rem}
If we multiply \eqref{eq:Fay} through by $\alpha\beta(\alpha+\beta)$ to remove negative powers of these variables, and expand as a power series in $\alpha$ and $\beta$, we recover classical elliptic function identities  upon equating coefficients. Indeed
\begin{align*}
[\alpha \beta]: \qquad \zeta(z-w) + \zeta(w-z) = 0
\end{align*}
is simply the fact that $\zeta$ is an odd function, while
\begin{align*}
[\alpha^2 \beta]: \qquad \left( \zeta(z-w) - \zeta(z) + \zeta(w) \right)^2 = \wp(z) + \wp(w) + \wp(z-w).
\end{align*}
\end{rem}

\begin{prop}\label{prop:to.E2}
The assignments
\begin{align*}
\psi(\nu_i) &= -\frac{1}{\ov{\tau} - \tau} [d\ov{x}_i], \\
\psi(\om_i) &= [dx_i], \\
\psi(\om_{ij}^{(0)}) &= [dx_i] - [dx_j], \\
\psi(\om_{ij}^{(1)}) &= G_{ij}, \\
\psi(\om_{ij}^{(m)}) &= 0, \quad m \geq 2,
\end{align*}
extend to a morphism of dg-algebras $\psi : \X_n \rightarrow E_2(n)$.
\end{prop}

Recall that $E_2(n)$ is defined as a quotient of $H^*(\ovX)^{\otimes n} \otimes G(n)$. The terms in square brackets $[\ldots]$ in the statement of the proposition stand for the classes in cohomology of the algebraic de Rham complex of $\ovX$, considered as elements of $H^*(\ovX)$. 

\begin{proof}
From the definition of $\psi$ it is immediate that
\begin{align*}
\psi(\Om_{ij}(\alpha)) = \al^{-1}([dx_i] - [dx_j])  + G_{ij}.
\end{align*}
Applying $\psi$ to \eqref{eq:Fay.rel.X} and equating coefficients in $\alpha$ and $\beta$ leads to a finite list of identities in the generators of $E_2(n)$, to be checked. By symmetry it suffices to examine coefficients of $\al^p\beta^q$ with $p \geq q$. Multiplying through by $\al\beta(\al+\beta)$ to eliminate negative exponents we obtain nontrivial identities in degrees $(p, q)$ equal to $(1, 0)$, $(1, 1)$, $(2, 0)$ and $(2, 1)$. The first of these is
\begin{align*}
(\psi(\om^{(0)}_{ij}) - \psi(\om^{(0)}_{j\ell})) \wedge \psi(\om^{(0)}_{j\ell}) = 0,
\end{align*}
which is easily seen to be satisfied. The identity corresponding to degree $(2, 0)$ is
\begin{align*}
G_{i\ell} \wedge ([dx_i] - [dx_\ell]) = 0,
\end{align*}
which follows from the defining relation
\begin{align}\label{eq:G.kills}
G_{ij} \otimes (\pi_i^*\eta - \pi^*_j\eta) = 0    
\end{align}
of $E_2(n)$, for $\eta = [dx]$. The identity corresponding to degree $(1, 1)$ also holds because of \eqref{eq:G.kills}. Finally the identity in degree $(2, 1)$ is precisely \eqref{eq:G.Gelfand}.

The differential relation
\begin{align*}
d(\psi(\om_{ij}^{(m+1)})) = (\psi(\nu_i) - \psi(\nu_j)) \wedge \psi(\om_{ij}^{(m)})
\end{align*}
is trivially verified for $m \geq 2$. For $m = 1$ it follows from the defining relation \eqref{eq:G.kills} of $E_2(n)$, for $\eta = [d\ov{x}]$. For $m=0$ the relation asserts that
\begin{align*}
d(G_{ij}) = -\frac{1}{\ov\tau - \tau} ([d\ov{x}_i] - [d\ov{x}_j]) \wedge ([dx_i] - [dx_j]),
\end{align*}
which is verified by refering to \eqref{eq:Gysin} and \eqref{eq:diag.class}.
\end{proof}

\begin{nolabel}\label{nolabel:filtr.Xi.Om.E2}
We introduce an increasing filtration $\wtil{G}$ on $\X_n$, compatible with the dg-algebra structure, by putting $\om_i, \om_{ij}^{(0)} \in \wtil{G}^0\X_n$, putting $\om_{ij}^{(m)} \in \wtil{G}^1\X_n$ for $m \geq 1$, and putting $\nu_i \in \wtil{G}^1\X_n$ (for all $1 \leq i, j \leq n$). If we set $G^p\X_n = \wtil{G}^{p+n}\X_n$ then the morphisms $\varphi$ and $\psi$ are compatible with the three filtrations $G$.
\end{nolabel}

\section{Filtration on the chiral complex: Definition}\label{sec:first.filtration}

\begin{nolabel}
In this section we define filtrations on the chiral complexes $C_\bullet$ and $Q_\bullet$, which blend the filtration $G$ on $\Gamma_n$ with a choice of good increasing filtration on $V$.
\end{nolabel}

\begin{defn}\label{def:C.filtr.def}
Let $V$ be a graded vertex algebra and $G$ a good increasing filtration on $V$, which we extend to $V^{\otimes n+1}$ in the natural way. We continue to denote also by $G$ the increasing filtration on $\Gamma_n$ defined in \S \ref{nolabel:explicit.C.Hodge}. The ordered chiral complex $C_\bullet$ was defined in \eqref{eq:C.complex.def}. The Hodge filtration on $C_\bullet$ is defined as follows: $\filg{p}C_n$ is the image in $C_n$ of
\begin{align*}
\sum_{p' + p'' \leq p} \filg{p'} \Gamma_n \otimes \filg{p''} V^{\otimes n+1} \subset \wtil{C}_n,
\end{align*}
under the natural surjection. The passage to the quotient makes sense since the filtrations $G$ on $\Gamma_n$ and $V$ are stable under $\partial_i$ and $T$, respectively.
\end{defn}

\begin{prop}\label{prop:G.is.filtr.on.C}
The subspaces $\filg{p}C_n$ identified in Definition \ref{def:C.filtr.def} define an increasing filtration on the complex $C_\bullet$.
\end{prop}

\begin{proof}
The differential in $C_\bullet$ is expressed in \eqref{eq:general.diff.def} as a linear combination of terms $d^{(ij)}$ and $p^{(i)}$, defined in \eqref{eq:general.diff.def.comp1} and \eqref{eq:general.diff.def.comp2}, each of which operates on $\Gamma_n \otimes V^{\otimes n+1}$ by taking the vertex algebra product between two of the $V$ tensor factors, extracting a residue at the corresponding divisor ($x_i=x_j$ in the case of $d^{(ij)}$ and $x_i=0$ in the case of $p^{(i)}$), and finally relabeling variables. For instance, ignoring relabeling, $p^{(i)}$ sends
\begin{align*}
f(x_1, \ldots, x_n) \cdot a^1 \cdots a^n \mid b
\end{align*}
to
\begin{align*}
\res_{x_i=0} f(x_1, \ldots, x_n) \cdot a^1 \cdots \widehat{a}^i \cdots a^n \mid Y(a^i, x_i) b.
\end{align*}
The residue in turn is expanded as in \eqref{eq:res.ser.exp} into a sum of functions $f^{(m)}$ of $n-1$ variables, times $a^i_{(m)}b$. Now we suppose $f \in \filg{-n+p}\Gamma_n$, and consider in turn the cases in which $f$ has a pole on the divisor $x_i=0$, and in which it does not. In the former case the functions $f^{(m)}$ have one pole fewer than $f$ does, i.e., $f^{(m)} \in \filg{-(n-1)+p-1}\Gamma_{n-1} = \filg{-n+p}\Gamma_{n-1}$, and so the filtration is preserved by $p^{(i)}$ since the filtration $G$ in $V$ is compatible with the $m^{\text{th}}$ product. In the latter case we have $f^{(m)} \in \filg{-n+p+1}\Gamma_{n-1}$, but since $f$ has no pole along $x_i=0$, only terms in which $m \geq 0$ appear. Since, if $a^i \in \filg{k}V$ and $b \in \filg{\ell}V$, we have $a^i_{(m)}b \in \filg{k+\ell-1}V$ for all $m \geq 0$, this summand of $p^{(i)}$ also respects the filtration. Thus the filtration is stable under $p^{(i)}$. Similar considerations apply to the $d^{(ij)}$, and so we conclude that the Hodge filtration is stable under $d$.
\end{proof}

\begin{nolabel}
Since $C_n$ is the quotient of $\Gamma_n \otimes V^{\otimes n+1}$ by relations of the form $f \otimes T_i a \equiv \partial_i f \otimes a$, every element of $C_n$ can be written as a combination of terms of the form $f \otimes a$ where $f$ runs over a set of representatives of $\Gamma_n / \left<\partial_i \Gamma_n\right>_{i=1,\ldots, n} \cong H^n(F(\ovX, n))$.  Thus for example $\dim H^2(F(\ovX, 2)) = 5$, with $\filg{-2}$ spanned by $1$, the quotient $\filg{-1}/\filg{-2}$ spanned by the images of $\wp(x_1)$ and $\zeta(x_1, x_2)$, and the quotient $\filg{0}/\filg{-1}$ spanned by the images of $\wp(x_1)\wp(x_2)$ and $\wp(x_1)\zeta(x_1, x_2)$.
\end{nolabel}

\begin{nolabel}
Finally, it is clear that the filtration $G$ on $C_\bullet$ is compatible with the passage to coinvariants, and so induces a filtration on the unordered chiral complex $Q_\bullet$ also.
\end{nolabel}

\section{The associated graded and Poisson homology}\label{sec:Poison.complex}

\begin{nolabel}
In \S\nopagebreak \ref{nolabel:deg0.HP} we carried out a brief analysis of $H_0(C_\bullet)$ using the Hodge filtration, relating it to $\HP_0(R_V)$. In this section we extend this analysis to higher homological degrees.
\end{nolabel}

\begin{nolabel}\label{nolabel:filtration.E}
Denote by $\pi$ the canonical map
\begin{align*}
\grg{}(\Gamma_n) \otimes \grg{}(V^{\otimes n+1}) \rightarrow \grg{}(\Gamma_n \otimes V^{\otimes n+1}) \rightarrow \grg{}(C_n),
\end{align*}
and put
\begin{align*}
\file{q}\grg{}(C_n) = \pi \bigoplus_{p \geq q} \grg{p}(\Gamma_n) \otimes \grg{}(V^{\otimes n+1}).
\end{align*}
\end{nolabel}

\begin{prop}\label{prop:E.is.a.filtration}
The subspaces $\file{q}\grg{}(C_n)$ define a decreasing filtration on the complex $\grg{}(C_\bullet)$, concentrated in negative degrees.
\end{prop}

\begin{proof}
Consider a representative $f \otimes a \otimes b$ of an element of $\grg{q} \otimes \grg{p}$. As in the proof of Proposition \ref{prop:G.is.filtr.on.C} we observe that $d(f \otimes a \otimes b)$ is a sum of two kinds of terms: (1) $\res(f) \otimes a_{(m)}b$ in which $\res(f) \in \grg{q}$ and $m \in \Z$, hence $a_{(m)}b \in \grg{p}$, and (2) $\res(f) \otimes a_{(m)}b$ in which $\res(f) \in \grg{q+1}$ and $m \in \Z_{\geq 0}$, hence $a_{(m)}b \in \grg{p-1}$. In particular there is no nontrivial component to the differential of the form $\grg{q} \otimes \grg{p} \rightarrow \grg{q-1} \otimes \grg{p+1}$. Thus the filtration $E$ is stable under the differential $d$.
\end{proof}

\begin{nolabel}
Let us write ${}'C_\bullet = \gre{}\grg{}C_\bullet$. This is a graded complex with nontrivial components ${}'C^p_n$ concentrated in degrees $(n, p)$ for $-n \leq p \leq 0$. We describe the homology of this complex in degree $(n, -n)$. We write $A$ for $\grg{}(V)$. The component ${}'C^{-n}_n$ is just $1 \cdot A^{\otimes n+1}$, and receives a differential from ${}'C^{-n}_{n+1}$. The component ${}'C^{-n}_{n+1}$ contains subspaces of the form $\wp(x_i) \cdot A^{\otimes n+2}$ and $\wp(x_i-x_j) \cdot A^{\otimes n+2}$ and $\zeta(x_i,x_j) \cdot A^{\otimes n+2}$, for $i \neq j$. (These subspaces span ${}'C^{-n}_{n+1}$, but are not linearly independent in general.) It is straightforward to compute
\begin{align}\label{eq:wp.xi.calc}
d\left( \wp(x_i) \cdot a^1 \cdots a^n \mid b \right) = (-1)^{n-i} 1 \cdot a^1 \cdots \widehat{a}^i \cdots a^n \mid a^i_{(-2)}b.
\end{align}
Indeed $\wp(x_i)$ has a pole only on $x_i = 0$, and so contributions to \eqref{eq:general.diff.def} involving residues along the divisors $x_i = x_j$ and $x_j = 0$, for $j \neq i$, contain a function with pole and so lie in $C_{n}^{-n+1}$. By definition these terms vanish in ${}'C_{n}$. The remaining term, the residue along the divisor $x_i=0$, is
\begin{align*}
(-1)^{n-i} 1 \cdot a^1 \cdots \widehat{a}^i \cdots a^n \mid \res_{x=0} \wp(x) Y(a^i, x)b.
\end{align*}
The leading term of the residue is $a^i_{(-2)}b$, with all subsequent terms $a^i_{(\geq 0)}b$ vanishing in the associated graded complex. Thus we arrive at \eqref{eq:wp.xi.calc}.

In the same way we have
\begin{align}\label{eq:wp.xixj.calc}
d\left( \wp(x_i-x_j) \cdot a^1 \cdots a^n \mid b \right) = (-1)^{n-i} 1 \cdot a^1 \cdots \widehat{a}^i \cdots a^i_{(-2)}a^j \cdots a^n \mid b.
\end{align}
Let us denote $J \subset A$ the linear span of terms of the form $a_{(-2)}b$ where $a, b \in V$. As in \S\nolinebreak \ref{nolabel:def.RV} above, the quotient $R = A / J$ is the Zhu Poisson algebra $R_V$. The images of \eqref{eq:wp.xi.calc} and \eqref{eq:wp.xixj.calc} inside ${}'C_n^{-n} \cong A^{\otimes n+1}$ are precisely the subspaces $A^{\otimes i} \otimes J \otimes A^{n-i}$, and the quotient of $A^{\otimes n+1}$ by their sum is $R^{\otimes n+1}$.

Upon passage to $\Sigma_n$-coinvariants \eqref{eq:symm.grp.action}, the tensor product $R^{\otimes n+1}$ above is further reduced to $R \otimes \bigwedge^n R$. We claim now that the quotient by the images of terms of the form $\zeta(x_i,x_j) \cdot A^{\otimes n+2}$ has the effect of reduction from $R \otimes \bigwedge^n R$ to $\bigwedge^n_R \Omega_R$. Indeed, as in the cases of \eqref{eq:wp.xi.calc} and \eqref{eq:wp.xixj.calc} above, we compute
\begin{align}\label{eq:zeta.xixj.calc}
d\left( \zeta(x_i, x_j) \cdot a^1 \cdots a^n \mid b \right) &=
-(-1)^{n-i} a^1 \cdots \widehat{a}^i \cdots a^n \mid a^i_{(-1)}b
+ (-1)^{n-j} a^1 \cdots \widehat{a}^j \cdots a^n \mid a^j_{(-1)}b \\
& + (-1)^{n-i} a^1 \cdots \widehat{a}^i \cdots a^i_{(-1)}a^j \cdots a^n \mid b.
\end{align}
After identifying the class in $H_n({}'C_{\bullet})^{-n}$ of $1 \cdot a^1 \cdots a^n \mid b$ with $b \cdot da^1 \wedge \cdots \wedge da^n$, the vanishing of \eqref{eq:zeta.xixj.calc} becomes the relation
\begin{align}\label{eq:zeta.Kah.rels}
\begin{split}
b \, da^1 \wedge \cdots da^i \cdots d(a^i a^j) \cdots \wedge da^n
 &= a^i b \, da^1 \wedge \cdots \widehat{da^i} \cdots da^j \cdots \wedge da^n \\
 &- (-1)^{i-j} a^j b \, da^1 \wedge \cdots da^i \cdots \widehat{da^j} \cdots \wedge da^n.
\end{split}
\end{align}
For $n=2$ this recovers the defining relation for the $R$-module of K\"{a}hler differentials $\Omega_R$, as in Definition \ref{defn:Kahler}. In general, the quotient by the relations \eqref{eq:zeta.Kah.rels} yields the $R$-module of K\"{a}hler differentials $\bigwedge^n_{R} \Omega_R$. In summary, we have proved
\end{nolabel}
\begin{prop}
Let $V$ be a finitely strongly generated graded vertex algebra, with associated Zhu Poisson algebra $R = R_V$. Then
\[
H_n({}'C_{\bullet})^{-n} \cong \bigwedge\nolimits^{\!n}_{R} \Omega_R.
\]
\end{prop}

\begin{nolabel}
The graded pieces of the homology of ${}'C_\bullet$ form the first page in a spectral sequence. The second page is equipped with differentials
\begin{align*}
H_n({}'C_{\bullet})^{-p} \rightarrow H_{n-1}({}'C_{\bullet})^{-p+1},
\end{align*}
which we now determine for $p=n$, following the standard recipe. As per the identification above, an element $a \, dx \in \Omega_R$ is represented by $1 \cdot x \mid a \in 1 \cdot A \otimes A$ which in turn is sent to $x_{(0)}a$ by the differential. This term vanishes in $\gr_E$ and reappears in the second page of the spectral sequence. In $R$ this term is simply $-\{a, x\}$. Similarly the image of $a \, dx \wedge dy \in \wedge^2 \Omega_R$ in $\Omega_R$ is
\begin{align*}
\{a, x\} \, dy - \{a, y\} \, dx - a \, d(\{x, y\}).
\end{align*}
A similar pattern continues to all orders, and we obtain the Lichnerowicz complex (defined below) of the Poisson algebra $R = R_V$.
\end{nolabel}

\begin{nolabel}
The Poisson homology $\HP_\bullet(P)$ of a Poisson algebra $P$ is the homology of an explicit complex $(C_\bullet(P), d)$, known as the Lichnerowicz complex, constructed from $P$. Let $\Omega_P$ denote the K\"{a}hler differentials, constructed using the commutative algebra structure of $P$. We set $C_0(P) = P$, $C_1(P) = \Omega_{P}$ and in general $C_n(P) = \wedge_{P}^n \Omega_P$, equipped with the following differentials:
\begin{align*}
d \left( a \, dx_1 \wedge \ldots \wedge dx_n \right)
&= \sum_{i=1}^n (-1)^{n-i} \{x_i, a\} \, dx_1 \wedge \ldots \wedge \widehat{dx}_i \wedge \ldots \wedge dx_n \\
& + \sum_{i < j} (-1)^{n-i}  a \, dx_1 \wedge \ldots \wedge \widehat{dx}_i \wedge \ldots \wedge d(\{x_i, x_j\}) \wedge \ldots \wedge dx_n.
\end{align*}
This is the Lichnerowicz complex, and its homology in degree $n$ is denoted $\HP_n(P)$, the Poisson homology of $P$. We summarise the findings of this section in the following proposition.
\end{nolabel}

\begin{thm}\label{thm:chiral.contains.HP}
The homology $H_n(\gr^G C_\bullet)$ carries a filtration $E$, concentrated in degrees $-n, \ldots 0$. The graded piece in degree $-n$ is identified with the Poisson homology of $R_V$, that is
\begin{align*}
\gre{-n} H_{n}(\grg{}C_\bullet) \cong \HP_n(R_V).
\end{align*}
\end{thm}

\begin{proof}
The filtration $E$ on $\grg{}C_\bullet$ yields a spectral sequence with first page $H({}'C_{\bullet})$, converging to $H_\bullet(\grg{}C)$. Comparing the explicit complex for $\HP_\bullet(R_V)$ with the discussion in the preceding sections, it is evident that $\gre{-n} H_{n}(\grg{}C_\bullet)$ is a subquotient of $\HP_n(R_V)$ (the latter being obtained on the second page of the spectral sequence in this degree). It remains to confirm that the spectral sequence collapses at this page. But this follows from the observation used in the proof of Proposition \ref{prop:E.is.a.filtration}.
\end{proof}

\section{Classical freedom and finiteness results}\label{sec:consequences}

\begin{nolabel}
In this section we focus attention on the structure of chiral homology in degree $1$. We have ${}'C_1 = {}'C_1^{-1} \oplus {}'C_1^{0}$, and we have seen above that the homology of ${}'C_\bullet$ in bidegree $(1, -1)$ is given by $\Omega_R$. We now analyse the homology of ${}'C_{1}^0$. The component ${}'C_{1}^0$ is spanned by the class of $\wp(x_1) \cdot A^{\otimes 2}$. As we saw in \eqref{eq:wp.xi.calc} the restriction of the differential to this component is given by
\begin{align}\label{eq:wp.diff}
d(\wp(x_1) \cdot a^1 \mid b) = a^1_{(-2)}b.
\end{align}

The component ${}'C_{2}^0$ is spanned by subspaces of the form $\wp(x_1) \wp(x_2) \cdot A^{\otimes 3}$ and $\wp(x_1) \zeta(x_1, x_2) \cdot A^{\otimes 3}$. It is easy to compute
\begin{align}\label{eq:wp.im.pp}
d(\wp(x_1) \wp(x_2) \cdot a^1 a^2 \mid b) = \wp(x_1) \cdot \left( a^1 \mid a^2_{(-2)}b - a^2 \mid a^1_{(-2)}b \right).
\end{align}
We now compute the image under $d$ of $\wp(x_2) \zeta(x_1, x_2) \cdot a^1 a^2 \mid b$. There are contributions from residues at the three divisors $x_1=0$, $x_2=0$ and $x_1=x_2$. At the divisor $x_1=0$, the only nonzero contribution is
\[
\left. \left( \res_{x_1=0} \wp(x_2) \zeta(x_1) \cdot a^2 \mid Y(a^1, x_1) b \right) \right|_{x_2=x_1} = \wp(x_1) \cdot a^2 \mid a^1_{(-1)}b.
\]
At the divisor $x_2=0$, we have two contributions. The first is
\[
\res_{x_2=0} \wp(x_2) \zeta(x_2) \cdot a^1 \mid Y(a^2, x_2) b = 1 \cdot a^1 \mid \left( a^2_{(-3)}b + \cdots \right).
\]
According to the definition of the filtration $E$, see Section \ref{nolabel:filtration.E}, this contribution vanishes in ${}'C_\bullet$, essentially because the function $1$ is two steps lower in the Hodge filtration than $\wp(x_2) \zeta(x_1, x_2)$. The second contribution is
\begin{align*}
& \res_{x_2=0} \wp(x_2) \left( \zeta(x_1-x_2) - \zeta(x_1) \right) \cdot a^1 \mid Y(a^2, x_2) b \\
&= \res_{x_2=0} \wp(x_2) \sum_{k \in \Z_{\geq 1}} \frac{1}{k!} \left[\partial^{k}\zeta\right](x_1) (-x_2)^k \cdot a^1 \mid Y(a^2, x_2) b.
\end{align*}
The product $x_2^k \wp(x_2)$ is regular when $k \geq 2$, and the corresponding summands vanish in the associated graded. On the other hand $\partial \zeta(x) = -\wp(x) - G_2(\tau)$, so the sum becomes
\begin{align*}
\wp(x_1) \res_{x_2=0} x_2 \wp(x_2) \cdot a^1 \mid Y(a^2, x_2) b
= \wp(x_1) \cdot a^1 \mid a^2_{(-1)}b.
\end{align*}
At the divisor $x_1 = x_2$, the only nonzero contribution is
\[
\left. \left( \res_{x_1=x_2} \wp(x_2) \zeta(x_1-x_2) \cdot Y(a^1, x_1-x_2) a^2 \mid b \right) \right|_{x_2=x_1} = \wp(x_1) \cdot a^1_{(-1)}a^2 \mid b.
\]
Combining these contributions, with appropriate signs, yields
\begin{align}\label{eq:wp.im.zp}
d(\wp(x_2) \zeta(x_1, x_2) \cdot a^1 a^2 \mid b) = \wp(x_1) \cdot \left( a^1 \mid a^2_{(-1)}b + a^2 \mid a^1_{(-1)}b - a^1_{(-1)}a^2 \mid b \right).
\end{align}

We recall that $A$ is a commutative algebra with product $a b = a_{(-1)}b$. The quotient of $\wp(x_1) \cdot A^{\otimes 2}$ by the vector subspace spanned by the right hand side of \eqref{eq:wp.im.zp} is therefore identified with the module $\Om_A$ of K\"{a}hler differentials of $A$, via $\wp(x_1) \cdot a \mid b \equiv b \, da$. Furthermore $A$ is a differential algebra with derivation $T : A \rightarrow A$ and the relation $a_{(-2)}b = (Ta) \cdot b$ holds in $A$. It follows that the morphism \eqref{eq:wp.diff} becomes
\[
\Om_A \rightarrow A, \qquad b \, da \mapsto b \, (Ta),
\]
and the right hand side of \eqref{eq:wp.im.pp} is the image of
\[
\bigwedge\nolimits^{\!2}_A \Om_A \rightarrow \Om_A, \qquad b \, da^1 \wedge da^2 \mapsto b \, (Ta^1) \, da^2 - b \, (Ta^2) \, da^1.
\]
We see, therefore, that the homology of ${}'C_\bullet$ in this degree recovers the Koszul homology $H_1(A, T)$ introduced in \S\nolinebreak \ref{nolabel:Koszul.complex} above.
\end{nolabel}

\begin{thm}
Let $V$ be a graded vertex algebra, finitely strongly generated and classically free. Let $E$ be a smooth elliptic curve and $\CA$ the chiral algebra on $E$ associated with $V$. Then
\begin{align*}
\dim \Hch_1(E, \CA) \leq \dim \HP_1(R_V).
\end{align*}
\end{thm}

\begin{proof}
As usual the dimension of the homology of an associated graded complex is an upper bound on the dimension of the homology of the complex itself. Thus
\begin{align*}
\dim \Hch_1(E, \CA) \leq \dim \HP_1(R_V) + \dim H_1({}'C_\bullet).
\end{align*}
But as we have just seen, $H_1({}'C_\bullet)$ is isomorphic to the Koszul homology in degree $1$ of the differential algebra $A$. Since $V$ is classically free, this Koszul homology vanishes by Proposition \ref{thm:classically.free.koszul}.
\end{proof}


\begin{thebibliography}{10}

\bibitem{AN}
Toshiyuki Abe and Kiyokazu Nagatomo.
\newblock Finiteness of conformal blocks over compact {R}iemann surfaces.
\newblock {\em Osaka J. Math.}, 40(1):375--391, 2003.

\bibitem{Ab.Steg}
Milton Abramowitz and Irene Stegun.
\newblock {\em Handbook of Mathematical Functions with Formulas, Graphs, and
  Mathematical Tables. National Bureau of Standards Applied Mathematics Series
  55. Tenth Printing.}
\newblock ERIC, 1972.

\bibitem{AEH23}
George Andrews, Jethro~van Ekeren, and Reimundo Heluani.
\newblock The singular support of the {I}sing model.
\newblock {\em International Mathematics Research Notices},
  2023(10):8800--8831, 2023.

\bibitem{Arakawa.C2}
Tomoyuki Arakawa.
\newblock A remark on the {$C_2$}-cofiniteness condition on vertex algebras.
\newblock {\em Mathematische Zeitschrift}, 270(1):559--575, Feb 2012.

\bibitem{Arakawa.Kawasetsu}
Tomoyuki Arakawa and Kazuya Kawasetsu.
\newblock Quasi-lisse vertex algebras and modular linear differential
  equations.
\newblock In {\em Lie Groups, Geometry,
  and Representation Theory: A Tribute to the Life and Work of Bertram
  Kostant}, Victor~G Kac and Vladimir Popov (editors), volume 326 of {\em Progress in Mathematics}, pages 41--57.
  Birkh\"{a}user, 2018.

\bibitem{BDHK19}
Bojko Bakalov, Alberto De~Sole, Reimundo Heluani, and Victor~G Kac.
\newblock An operadic approach to vertex algebra and {P}oisson vertex algebra
  cohomology.
\newblock {\em Japanese Journal of Mathematics}, 14:249--342, 2019.

\bibitem{BDHKV21}
Bojko Bakalov, Alberto De~Sole, Reimundo Heluani, Victor~G Kac, and Veronica
  Vignoli.
\newblock Classical and variational {P}oisson cohomology.
\newblock {\em Japanese Journal of Mathematics}, 16(2):203--246, 2021.

\bibitem{BDK20}
Bojko Bakalov, Alberto De~Sole, and Victor~G Kac.
\newblock Computation of cohomology of {L}ie conformal and {P}oisson vertex
  algebras.
\newblock {\em Selecta Mathematica}, 26:1--51, 2020.

\bibitem{BDK21}
Bojko Bakalov, Alberto De~Sole, and Victor~G Kac.
\newblock Computation of cohomology of vertex algebras.
\newblock {\em Japanese Journal of Mathematics}, 16:81--154, 2021.

\bibitem{BDKV21}
Bojko Bakalov, Alberto~De Sole, Victor~G Kac, and Veronica Vignoli.
\newblock Poisson vertex algebra cohomology and differential {H}arrison
  cohomology.
\newblock In {\em Representation Theory, Mathematical Physics, and Integrable
  Systems: In Honor of Nicolai Reshetikhin}, pages 39--69. 2021.

\bibitem{BD}
Alexander Beilinson and Vladimir Drinfeld.
\newblock {\em Chiral algebras}, volume~51 of {\em Colloquium Publications}.
\newblock American Mathematical Society, 2004.

\bibitem{BL}
Francis Brown and Andrey Levin.
\newblock Multiple elliptic polylogarithms. \\
\newblock Retrieved from \texttt{https://arXiv:1110.6917}, 2011.

\bibitem{DGT}
Chiara Damiolini, Angela Gibney, and Nicola Tarasca.
\newblock On factorization and vector bundles of conformal blocks from vertex
  algebras.
\newblock {\em Annales scientifiques de l’\'{E}cole normale sup\'{e}rieure},
  2022.
\newblock To appear.

\bibitem{DSK.variational}
Alberto De~Sole and Victor~G Kac.
\newblock The variational {P}oisson cohomology.
\newblock {\em Japanese Journal of Mathematics}, 8(1):1--145, 2013.

\bibitem{Deligne}
Pierre Deligne.
\newblock Th{\'e}orie de {H}odge: {II}.
\newblock {\em Publications Math{\'e}matiques de l'IH{\'E}S}, 40:5--57, 1971.

\bibitem{DM06}
Chongying Dong and Geoffrey Mason.
\newblock Integrability of {$C_2$}-cofinite vertex operator algebras.
\newblock {\em International Mathematics Research Notices}, 2006:80468, 2006.

\bibitem{EH21}
Jethro~van Ekeren and Reimundo Heluani.
\newblock Chiral homology of elliptic curves and the {Z}hu algebra.
\newblock {\em Communications in Mathematical Physics}, 386(1):495--550, 2021.

\bibitem{FBZ}
Edward Frenkel and David Ben-Zvi.
\newblock {\em Vertex algebras and algebraic curves}.
\newblock Number~88 in Mathematical surveys and monographs. American
  Mathematical Soc., 2004.

\bibitem{kac.book}
Victor~G Kac.
\newblock {\em Vertex algebras for beginners}.
\newblock Number~10 in University lecture series. American Mathematical Soc.,
  1998.

\bibitem{KL1-4}
David Kazhdan and George Lusztig.
\newblock Tensor structures arising from affine {L}ie algebras. {I - IV}.
\newblock {\em Journal of the American Mathematical Society}, 6(4):905--947,
  1993.

\bibitem{Li.filtration}
Haisheng Li.
\newblock Vertex algebras and vertex {P}oisson algebras.
\newblock {\em Communications in Contemporary Mathematics}, 6(01):61--110,
  2004.

\bibitem{Li.abelianizing}
Haisheng Li.
\newblock Abelianizing vertex algebras.
\newblock {\em Communications in mathematical physics}, 259:391--411, 2005.

\bibitem{Loday.93}
Jean-Louis Loday.
\newblock Une version non commutative des algebres de {L}ie: les algebres de
  {L}eibniz.
\newblock {\em Recherche Coop{\'e}rative sur Programme n{${}^\circ$} 25},
  44:127--151, 1993.

\bibitem{Milnor-Stasheff}
John Milnor and James Stasheff.
\newblock {\em Characteristic classes}.
\newblock Number~76 in Annals of Mathematics Studies. Princeton university
  press, 1974.

\bibitem{Saito}
Kyoji Saito.
\newblock Theory of logarithmic differential forms and logarithmic vector
  fields.
\newblock {\em Journal of the Faculty of Science, the University of Tokyo.
  Sect. 1 A, Mathematics}, 27(2):265--291, 1980.

\bibitem{stacks-project}
The {Stacks Project Authors}.
\newblock \textit{Stacks Project}.
\newblock \url{https://stacks.math.columbia.edu}, 2018.

\bibitem{Totaro}
Burt Totaro.
\newblock Configuration spaces of algebraic varieties.
\newblock {\em Topology}, 35(4):1057--1067, 1996.

\bibitem{TUY}
Akihiro Tsuchiya, Kenji Ueno, and Yasuhiko Yamada.
\newblock Conformal field theory on universal family of stable curves with
  gauge symmetries.
\newblock In {\em Integrable Systems in Quantum Field Theory and Statistical
  Mechanics}, volume~19 of {\em Advanced Studies in Pure Mathematics}, pages
  459--566. Mathematical Society of Japan, 1989.

\bibitem{Voisin-vol1}
Claire Voisin.
\newblock Hodge theory and complex algebraic geometry {I}.
\newblock {\em Cambridge Studies in Advanced Mathematics}, 76(11), 2002.
\newblock Translated from the French original by Leila Schneps.

\bibitem{zhu96}
Yongchang Zhu.
\newblock Modular invariance of characters of vertex operator algebras.
\newblock {\em Journal of the American Mathematical Society}, 9(1):237--302,
  1996.

\end{thebibliography}
\end{document}